\documentclass{imanum}

\usepackage{latexsym}
\usepackage{mathrsfs}
\usepackage{amsfonts}
\usepackage{amssymb}
\usepackage{epsfig}       
\usepackage{subfigure}    
\usepackage{graphicx}
\usepackage{float}
\usepackage{multirow}
\usepackage{slashbox}
\usepackage{bm}

\received{}
\revised{8 April 2014}

\begin{document}

\title{Minimum $n$-Rank Approximation via Iterative Hard Thresholding}
\shorttitle{Minimum $n$-Rank Approximation via IHT}

\author{%
{\sc Min Zhang\thanks{Email: minzhang@tju.edu.cn},
Lei Yang\thanks{Email: ylei@tju.edu.cn}
and Zheng-Hai Huang\thanks{Corresponding author. He is also with the Center for Applied Mathematics of Tianjin
University. Email: huangzhenghai@tju.edu.cn}} \\[2pt]
Department of Mathematics, School of Science,
Tianjin University, Tianjin 300072, P.R. China}
\shortauthorlist{M. Zhang \emph{et al.}}

\maketitle

\begin{abstract}
{The problem of recovering a low $n$-rank tensor is an extension of sparse recovery problem from the low dimensional
space (matrix space) to the high dimensional space (tensor space) and has many applications in computer vision and graphics
such as image inpainting and video inpainting. In this paper, we consider a new tensor recovery model, named as minimum
$n$-rank approximation (MnRA), and propose an appropriate iterative hard thresholding algorithm with giving the upper bound
of the $n$-rank in advance. The convergence analysis of the proposed algorithm is also presented. Particularly, we show that for the noiseless case, the linear convergence with rate $\frac{1}{2}$ can be obtained for the proposed algorithm under proper conditions.
Additionally, combining an effective heuristic for determining $n$-rank, we can also apply the proposed algorithm to solve MnRA when $n$-rank is unknown in advance. Some preliminary numerical results on randomly generated and real low $n$-rank tensor
completion problems are reported, which show the efficiency of the proposed algorithms.}
{iterative hard thresholding; low-$n$-rank tensor recovery; tensor completion; compressed sensing}
\end{abstract}

\section{Introduction}
\label{sec;introduction}

The problem of recovering an unknown low-rank matrix $\hat{X} \in \mathbb{R}^{m\times n}$ from the linear constraint
$\mathscr{A}(\hat{X})=\bm{b}$, where $\mathscr{A}:~\mathbb{R}^{m\times n}\rightarrow\mathbb{R}^p$ is the linear
transformation and $\bm{b} \in \mathbb{R}^p$ is the measurement, has been an active topic of recent research
with a range of applications including collaborative filtering (the Netflix probolem) \citep{gnot92},
multi-task learning \citep{aep08}, system identification \citep{lv09}, and sensor localization \citep{blwy06}. One method to
solve this inverse problem is to solve the matrix rank minimization problem:
\begin{eqnarray}\label{pm1}
\min_{X\in \mathbb{R}^{m\times n}}~\mbox{\rm rank}(X) \quad \mbox{\rm s.t.}\quad \mathscr{A}(X)={\bm b},
\end{eqnarray}
which becomes a mathematical task of minimizing the rank of $X$ such that it satisfies the linear constraint.
With the application of nuclear norm which is the tightest convex approach to the rank function, one can relax the non-convex
NP-hard problem (\ref{pm1}) to a tractable, convex one \citep[see][]{bmp2010,eb2009}.
An alternative model of this inverse problem is the minimum rank approximation problem:
\begin{eqnarray}\label{pm}
\min_{X\in \mathbb{R}^{m\times n}}~\|\bm{b}-\mathscr{A}(X)\|_2 \quad \mbox{\rm s.t.}\quad \mbox{\rm rank}(X)\leq r,
\end{eqnarray}
where $r=\mathrm{rank}(\hat{X})$ is known in advance, and $\hat{X}$ is the true data to be reconstructed. The model in (\ref{pm}) has been widely studied in the literature \citep[see][]{hh09,kom10,lb10,rs2009,dmk2011,bl09}. In fact, this formulation can not only work for the exact recovery case ($\mathscr{A}(\hat{X})=\bm{b}$), but also suit for the noisy case ($\bm{b}=\mathscr{A}(\hat{X})+\epsilon$), where $\epsilon$ denotes the noise by which the measurements are corrupted. Although the model (\ref{pm}) is based on a priori knowledge of the rank of
$\hat{X}$, an incremental search over $r$, which increases the complexity of the solution by at most factor $r$, can be applied when the minimum rank $r$ is unknown. Particularly, if an upper bound on $r$ is available, we can use a bisection search over $r$ since the minimum of (\ref{pm}) is monotonously decreasing in $r$. Then the factor can reduce to $\log r$.
Several effective algorithms based on (\ref{pm}) have
been proposed, such as OPTSPACE \citep{rs2009}, Space Evolution and Transfer (SET) \citep{dmk2011},
Atomic Decomposition for Minimum Rank Approximation (ADMiRA) \citep{lb10} and the Iterative Hard Thresholding (IHT)
introduced in \citep{gm10}. Among these algorithms, iterative hard thresholding algorithm is an easy-to-implement and fast method, which also shows the strong performance guarantees available with methods based on convex
relaxation.

Recently, many researchers focus on the recovery problem in the high dimensional space, which has many applications in
computer vision and graphics such as image inpainting \citep{bmgv2000} and video inpainting. More specifically, by using the
$n$-rank as a sparsity measure of a tensor (or multidimensional array), this inverse problem can be transformed into
the mathematical task of recovering an unknown low $n$-rank tensor $\hat{\mathcal{X}} \in \mathbb{R}^{n_1 \times \cdots \times n_N}$
from its linear measurements $\mathscr{A}(\hat{\mathcal{X}})=\bm{b}$ via a given linear transformation $\mathscr{A}:{\mathbb{R}}^{n_1\times n_2\times \ldots \times n_N } \rightarrow \mathbb{R}^p$ with $p\leq\prod^{N}_{i=1}n_i$.
Some related works can be found in \citet{gry11}, \citet{jppj2009}, \citet{mlj2010}, \citet{mqlj2011} and
\citet{yhs12}. In all these studies,
the authors mainly discussed the following tensor recovery model:
\begin{eqnarray}\label{p11}
\min_{\mathcal{X}}~\sum_{i=1}^N w_i\mbox{\rm rank}({\mathcal{X}}_{<i>})\quad \mbox{\rm s.t.}\quad
\mathscr{A}({\mathcal{X}})=\bm{b},
\end{eqnarray}
where $\mathcal{X} \in \mathbb{R}^{n_1 \times \cdots \times n_N}$ is the decision variable, ${\mathcal{X}}_{<i>}$
is the mode-$i$ unfolding (the notation will be given in Section 2) of $\mathcal{X}$, $w_i$'s are the weighted parameters which satisfy $0\leq w_i\leq 1$ and $\sum_{i=1}^N w_i=1$. Note that (\ref{p11}) can be regarded
as an extension of (\ref{pm1}) in the high dimensional space ${\mathbb{R}}^{n_1\times n_2\times \ldots \times n_N }$
and it is a difficult non-convex problem due to the combination nature of the function rank($\cdot$).
In order to solve it, the common method is replacing rank($\cdot$) by its convex envelope to get a convex tractable approximation
and developing effective algorithms to solve the convex approximation, including FP-LRTC (fixed point continuation method for low
$n$-rank tensor completion) \citep{yhs12}, TENSOR-HC (hard completion) \citep{mqlj2011}, and ADM-TR(E) (alternative direction method algorithm for low-$n$-rank tensor recovery) \citep{gry11}. Additionally, \citet{zh12}
investigated the exact recovery conditions for the low $n$-rank tensor recovery problems via its convex
relaxation. And lately, \citep{gq2014} studied the problem of robust low $n$-rank tensor recovery in a convex optimization framework, drawing upon
recent advances in robust Principal Component Analysis and tensor completion.

In this paper, we consider a new alternative recovery model extended from problem (\ref{pm}), which is called as
{\em minimum $n$-rank approximation} (MnRA):
\begin{eqnarray}\label{p1}
\min_{\mathcal{X}\in \mathbb{R}^{n_1 \times \cdots \times n_N} }~\|\mathscr{A}({\mathcal{X}})-\bm{b}\|^2_2\quad \mbox{\rm s.t.}
\quad \mbox{\rm rank}({\mathcal{X}}_{<i>})\leq r_i~~\forall~i,
\end{eqnarray}
where $(r_1,r_2,\cdots,r_N)$ is the $n$-rank of the ture data $\hat{\mathcal{X}}$ to be restored.
Note that this formulation has not been discussed in tensor space in the literature to our knowledge and it
also includes both the noisy case ($\mathscr{A}(\hat{\mathcal{X}})+\epsilon=\bm{b}$) and noiseless case ($\mathscr{A}(\hat{\mathcal{X}})=\bm{b}$).
One of its special cases is the {\em low $n$-rank tensor
completion} (LRTC) problem:
\begin{eqnarray}\label{p-add1}
\min_{\mathcal{X}\in \mathbb{R}^{n_1 \times \cdots \times n_N}}~\|\mathcal{X}_{\Omega}-\mathcal{M}_{\Omega}\|^2_F \quad \mbox{\rm s.t.}
\quad \mbox{\rm rank}({\mathcal{X}}_{<i>})\leq r_i~~\forall~i,
\end{eqnarray}
where $\mathcal{X}$, $\mathcal{M}$ are $N$-way tensors with identical size in each mode, and $\mathcal{X}_{\Omega}$
(or $\mathcal{M}_{\Omega}$) denotes the tensor whose $(i_1,i_2,\cdots,i_N)$-th component
equal to $\mathcal{X}_{i_1 i_2 \cdots i_N}$ (or $\mathcal{M}_{i_1 i_2 \cdots i_N}$) if $(i_1,i_2,\cdots,i_N) \in \Omega$
and zero otherwise. To solve (\ref{p1}), we propose an iterative hard thresholding algorithm, which is easy to implement
and very fast. Particularly, we prove that for the noiseless case the iterative sequence generated by the proposed algorithm is globally linearly convergent to the true data $\hat{\mathcal{X}}$ with the rate $\frac{1}{2}$ under
some conditions, while for the noisy case the distance between the iterative sequence and the true data $\hat{\mathcal{X}}$ is decreased quickly associated with the noise $\epsilon$. Additionally, combining an effective heuristic
for determining $n$-rank, we can also apply the proposed algorithm to solve MnRA
when $n$-rank of $\hat{\mathcal{X}}$ is unknown in advance. Some preliminary numerical
results are reported and demonstrate the efficiency of the proposed algorithms.

The rest of our paper is organized as follows. In Section 2, we first briefly introduce some preliminary knowledge of tensor.
Then, we propose an iterative hard thresholding algorithm to solve the minimum $n$-rank approximation problem
in Section 3 and the convergence analysis of the proposed algorithm will be presented in Section 4. In Section 5 and Section 6, we
give some implementation details and report some preliminary numerical results for low $n$-rank tensor completion,
respectively. Conclusions are given in the last section.

\section{Preliminary knowledge}

In this section, we briefly introduce some essential nomenclatures and notations used in this paper;
and more details can be found in \citet{tb2008}. Scalars are denoted by lowercase letters, e.g.,
$a,b,c,\cdots$; vectors by bold lowercase letters, e.g., $\bm{a},\bm{b},\bm{c},\cdots$;
and matrices by uppercase letters, e.g., $A,B,C,\cdots$. An \textit{N}-way tensor is denoted as
$\mathcal{X}\in\mathbb{R}^{n_1 \times \ldots \times n_N}$, whose elements are denoted as
$x_{j_1 \cdots j_k \cdots j_N}$, where $1\leqslant j_k \leqslant n_k$ and $1\leqslant k\leqslant N$.
Let us denote tensor space by $\mathbb{T}$ for convenience, i.e., ${\mathbb{T}}:={\mathbb{R}}^{n_1\times n_2
\times \ldots \times n_N }$. Then, for any ${\mathcal{X}}, {\mathcal{Y}}\in {\mathbb{T}}$, the inner product
is defined as
\begin{eqnarray*}
\langle{\mathcal{X}},{\mathcal{Y}}\rangle=\sum_{j_1=1}^{n_1}\sum_{j_2=1}^{n_2}\cdots\sum_{j_N=1}^{n_N}
{\mathcal{X}}_{j_1 j_2 \ldots j_N}{\mathcal{Y}}_{j_1 j_2 \ldots j_N}.
\end{eqnarray*}
Obviously, the tensor space ${\mathbb{T}}$ becomes a Hilbert space with the above definition of the inner
product, and the corresponding Frobenius-norm is $\|{\mathcal{X}}\|_F=\sqrt{\langle{\mathcal{X}}, {\mathcal{X}}\rangle}$.

The mode-$i$ fibers are all vectors $x_{j_1\ldots j_{i-1}:j_{i+1}\ldots j_N}$ obtained by fixing the indexes
of $\{j_1,\ldots j_N\}\backslash j_i$, which are analogue of matrix rows and columns. The mode-$i$ unfolding
of ${\mathcal{X}}\in{\mathbb{T}}$, denoted by ${\mathcal{X}}_{<i>}$, arranges the mode-$i$ fibers to be the
columns of the resulting matrix. The tensor element $(j_1,j_2,\ldots,j_N)$ is mapped to the matrix element
$(j_i,l)$, where
\begin{eqnarray*}
l=1+\sum_{k=1,\;k\neq i}^N(j_k-1)L_k~~\mbox{\rm with}~~L_k=\prod_{j=1,\;j\neq i}^{k-1}n_j,
\end{eqnarray*}
which infers ${\mathcal{X}}_{<i>}\in \mathbb{R}^{n_i\times T_i}$, where $T_i=\prod_{k=1,k\neq i}^N n_k$.
We also follow \citet{gry11} to define the $n$-rank of a tensor ${\mathcal{X}}\in {\mathbb{T}}$ by
\begin{eqnarray*}
n\mathrm{-}\mbox{\rm rank}({\mathcal{X}})=
(\mbox{\rm rank}({\mathcal{X}}_{<1>}),\mbox{\rm rank}({\mathcal{X}}_{<2>}),\cdots,\mbox{\rm rank}({\mathcal{X}}_{<N>})).
\end{eqnarray*}
In the following parts of this paper, we say ${\mathcal{X}}$ is an ($r_1, r_2,..., r_N$)-rank tensor, if for any $i$, the
rank of mode-$i$ unfolding of ${\mathcal{X}}$ is not greater than $r_i$, i.e., $\mbox{\rm rank}({\mathcal{X}}_{<i>})
\leq r_i$ for all $i$. It should be pointed out that this definition is different from the notation of a ``rank-($r_1, r_2,..., r_N$) tensor" in \citet{lbj2000}, which represents a tensor with the rank of each mode-$i$ unfolding is exactly $r_i$. The best ($r_1, r_2, ..., r_N$)-rank approximation $\widetilde{{\mathcal{X}}}$ of a tensor ${\mathcal{X}}$ is defined as the following:
$$\widetilde{{\mathcal{X}}}=\arg\min\{ \|{\mathcal{Y}}-{\mathcal{X}}\|_F:~{\mathcal{Y}}~ \mbox{\rm is a (}r_1, r_2, ..., r_N\mbox{\rm )-rank tensor}\}.$$

The $i$-mode (matrix) product of a tensor $\mathcal{X}\in\mathbb{R}^{n_1 \times \cdots \times n_N}$ with a matrix $U \in \mathbb{R}^{L \times n_i}$ is denoted by
$\mathcal{X} \times_{i} U$ and is of size $n_1 \times \cdots \times n_{i-1} \times L \times n_{i+1} \times \cdots \times n_N$. It can be expressed in terms of unfolded tensors:
\begin{eqnarray*}
\mathcal{Y} = \mathcal{X} \times_{i} U \quad \Longleftrightarrow  \quad Y_{<i>}=U X_{<i>}.
\end{eqnarray*}

Additionally, for any transformation $\mathscr{A}$, $\|\mathscr{A}\|$ denotes the operator norm of the transformation $\mathscr{A}$;
and for any vector $x$, we use Diag$(x)$ to denote a diagonal matrix with its $i$-th diagonal element being $x_i$.


\section{Iterative hard thresholding for low $n$-rank tensor recovery}

In this section, we will derive an iterative hard thresholding algorithm to solve problem (\ref{p1}). As a fast and efficient algorithm,
iterative hard thresholding algorithm has been
widely applied in various fields. \citet{bd08} and \citet{p09} first independently
proposed the iterative hard thresholding algorithm to solve the compressed sensing recovery problem.
Later, Blumensath and Davies presented a theoretical analysis of the iterative hard thresholding algorithm when applied to the compressed
sensing recovery problem in \citet{bd09}. Through the analysis, they showed that the simple iterative hard thresholding
algorithm has several good properties, including near-optimal error guarantees, robustness to observation noise, short memory requirement and low computational complexity. Also, it requires a fixed number of iterations and its performance guarantees are uniform. Recently, \citet{b12} used acceleration methods of choosing the step-size appropriately to improve the convergence speed of the iterative hard thresholding algorithm.

When it came to matrix space from vector space, \citet{gm10} studied
the convergence/ recoverability properties of the fixed point continuation algorithm and its variants for matrix
rank minimization. Particularly, in \citet{gm10}, the authors proposed an iterative hard thresholdinging algorithm
and discussed its convergence. At each iteration of their iterative hard thresholding algorithm, the authors
first performed a gradient step
$$Y^{k+1}:=X^k-{\mathscr{A}}^*({\mathscr{A}}(X^k)-\bm{b}),$$
where $X^{k}$ denotes the $k$-th iteration of $X$, $Y^{k+1}$ denotes the ($k+1$)-th iteration of $Y$ and ${\mathscr{A}}^*$
is the adjoint operator of ${\mathscr{A}}$ that is a linear transformation operating from $\mathbb{R}^p$ to $\mathbb{T}$. Then,
they applied hard thresholding operator to the singular values of $Y^{k+1}$, i.e., they only kept the largest $r$ singular
values of $Y^{k+1}$, to get $X^{k+1}$. It is easy to see that $X^{k+1}$ is actually the best $r$-rank approximation
to $Y^{k+1}$. More specifically, by using $R_r(X)$ to denote the hard thresholding operator with threshold $r$ for $X$,
the iterative scheme of their algorithm is as follows:
\begin{eqnarray*}
\left\{\begin{array}{l} Y^{k+1}=X^k-{\mathscr{A}}^*({\mathscr{A}}(X^k)-\bm{b}),\\
X^{k+1}=R_r(Y^{k+1}).\end{array}\right.
\end{eqnarray*}
Lately, in \citet{kc13}, they studied acceleration schemes via memory-based techniques and randomized,
$\epsilon$-approximate matrix projections to decrease the computational costs in the recovery process.
In this paper, inspired by the work of \citet{gm10}, we will develop an iterative hard thresholding algorithm for minimum $n$-rank approximation (\ref{p1}). In the following, we will do some theoretical analysis of problem (\ref{p1}) in order to derive the iterative scheme of our algorithm.

Firstly, we consider the objective function $f({\mathcal{X}}):=\|{\mathscr{A}}({\mathcal{X}})
-\bm{b}\|_2^2$ in (\ref{p1}). Similar to that in \citet{bd08}, we introduce a surrogate objective function $F:~\mathbb{T}\otimes\mathbb{T}\otimes\cdots\otimes\mathbb{T}
\rightarrow \mathbb{R}$ instead of function $f$:
\begin{eqnarray}\label{sfunction}
\begin{array}{lc}
F(\mathcal{Z}_0,\mathcal{Z}_1,\mathcal{Z}_2,\ldots,\mathcal{Z}_N)=\|\mathscr{A}(\mathcal{Z}_0)-\bm{b}\|_2^2
+\frac{1}{\tau}\sum \limits_{i=1}^N w_i\|\mathcal{Z}_0-\mathcal{Z}_i\|_F^2-\sum \limits_{i=1}^N w_i\|{\mathscr{A}}(\mathcal{Z}_0)-{\mathscr{A}}({\mathcal{Z}}_i)\|_2^2,
\end{array}
\end{eqnarray}
where $\tau>0$, $w_i\in[0,1]$, $\sum_{i=1}^N w_i=1$ and $\mathcal{Z}_0,\mathcal{Z}_1,\ldots,\mathcal{Z}_N \in \mathbb{T}$ are auxiliary variables in the domain of function $F$.
It is obvious that $F({\mathcal{X}},{\mathcal{X}},\ldots,{\mathcal{X}})=f({\mathcal{X}})$ and if $\|\mathscr{A}\|^2\leq \frac{1}{\tau}$, $f({\mathcal{X}})\leq F({\mathcal{X}},\mathcal{Z}_1,\mathcal{Z}_2,\ldots$, $\mathcal{Z}_N)$ for all ${\mathcal{X}}\in \mathbb{T}$, where $\|\mathscr{A}\|$ denotes the operator norm of linear operator $\mathscr{A}$. So, function $F$ is said to majorize $f$.

Let ${\mathcal{X}}^k$ be the $k$-th iteration and the $(k+1)$-th iteration ${\mathcal{X}}^{k+1}$ be the minimum of the function $F$ by setting its later $N$ variables to $X^k$, i.e., ${\mathcal{X}}^{k+1}=\mathop{\mathrm{argmin}}\limits_{\mathcal{X}} F({\mathcal{X}},
{\mathcal{X}}^k, {\mathcal{X}}^k,\ldots,{\mathcal{X}}^k)$. If $\|\mathscr{A}\|^2\leq \frac{1}{\tau}$, we have
\begin{eqnarray*}
f({\mathcal{X}}^{k+1}) &=&\|\mathscr{A}({\mathcal{X}}^{k+1})-\bm{b}\|_2^2\\
&\leq& \|\mathscr{A}({\mathcal{X}}^{k+1})-\bm{b}\|_2^2+(\frac{1}{\tau}\|{\mathcal{X}}^{k+1}-{\mathcal{X}}^k\|_F^2-\|{\mathscr{A}}({\mathcal{X}}^{k+1})-{\mathscr{A}}({\mathcal{X}}^k)\|_2^2)\\
&=& F({\mathcal{X}}^{k+1},{\mathcal{X}}^k,{\mathcal{X}}^k,...,{\mathcal{X}}^k)\\
&\leq& F({\mathcal{X}}^{k},{\mathcal{X}}^k,{\mathcal{X}}^k,...,{\mathcal{X}}^k)\\
&=& f({\mathcal{X}}^{k}),
\end{eqnarray*}
where the first inequality follows from the assumption that $\|\mathscr{A}\|^2\leq \frac{1}{\tau}$, and the second inequality
follows from that ${\mathcal{X}}^{k+1}$ is the minimum of $F({\mathcal{X}},{\mathcal{X}}^k,{\mathcal{X}}^k,
\ldots,{\mathcal{X}}^k)$. Thus, it can be clearly seen that if $\|\mathscr{A}\|^2\leq \frac{1}{\tau}$, fixing the latter $N$ variables in (\ref{sfunction}) and optimizing (\ref{sfunction}) with respect to the first variable will then decrease the value of the original objective function $f$. In other words, if $\|\mathscr{A}\|^2\leq \frac{1}{\tau}$, the iterative scheme solving problem (\ref{p1}) could be:
$${\mathcal{X}}^{k+1}=\mathop{\mathrm{argmin}}\limits_{\mathcal{X}\in \mathscr{C}} ~F({\mathcal{X}},{\mathcal{X}}^k,\ldots,{\mathcal{X}}^k),$$
where $\mathscr{C}$ denotes the set $\{\mathcal{X} | \mbox{\rm rank}(\mathcal{X}_{<i>})\leq r_i, ~i=1,2,\cdots,N \}$.

Note that, $F({\mathcal{X}},{\mathcal{X}}^k,\ldots,{\mathcal{X}}^k)$ can be written as:
\begin{eqnarray*}
\begin{array}{lc}
F(\mathcal{X},\mathcal{X}^k,\mathcal{X}^k,...,\mathcal{X}^k)
=\frac{1}{\tau}\left(\|\mathcal{X}\|_F^2-2\langle\mathcal{X}, \mathcal{X}^k-\tau\mathscr{A}^*(\mathscr{A}
(\mathcal{X}^k)-\bm{b})\rangle\right)    \vspace{1mm}  \\
~~~~~~~~~~~~~~~~~~~~~~~~~~~~~~~~~~~~~~~~~~~~+\|\bm{b}\|_2^2-\|\mathscr{A}(\mathcal{X}^k)\|_2^2-\frac{1}{\tau}\|\mathcal{X}^k\|_F^2.
\end{array}
\end{eqnarray*}
Then, it is easy to see that the solution of the following problem (without the constraints $\mbox{\rm rank}({\mathcal{X}}_{<i>})\leq r_i$ for any $i \in \{1,2,\cdots,N\}$):
\begin{eqnarray*}
{\mathcal{X}}^{k+1}=\mathop{\mathrm{argmin}}\limits_{\mathcal{X}\in \mathbb{T}}~ F({\mathcal{X}},{\mathcal{X}}^k,\ldots,{\mathcal{X}}^k)
\end{eqnarray*}
is
\begin{eqnarray*}
{\mathcal{X}}^{k+1}=\mathcal{X}^k-\tau\mathscr{A}^*(\mathscr{A}(\mathcal{X}^k)-\bm{b}),
\end{eqnarray*}
and the value of $F$ at this time is equal to
\begin{eqnarray*}
-\frac{1}{\tau}\|\mathcal{X}^{k+1}\|_F^2 + \|\bm{b}\|_2^2-\|{\mathscr{A}}(\mathcal{X}^k)\|_2^2-\frac{1}{\tau}\|\mathcal{X}^k\|_F^2.
\end{eqnarray*}
Therefore, the minimum of $F(\mathcal{X},\mathcal{X}^k,\mathcal{X}^k,\ldots,\mathcal{X}^k)$ with the constraints $\mbox{\rm rank}({\mathcal{X}}_{<i>})\leq r_i$ for any $i \in \{1,2,\cdots,N\}$ is then achieved at the best rank-$(r_1, r_2,...,r_N)$
approximation of $\mathcal{X}^{k+1}$, i.e.,
$${\mathcal{X}}^{k+1}=\arg\min_{\mathcal{X}\in \mathscr{C}} F({\mathcal{X}},{\mathcal{X}}^k,\ldots,{\mathcal{X}}^k)
=H_r\left(\mathcal{X}^k-\tau\mathscr{A}^*(\mathscr{A}(\mathcal{X}^k)-\bm{b})\right),$$
where $H_r(\mathcal{Y})$ means the best rank-$(r_1, r_2,...,r_N)$ approximation of $\mathcal{Y}$.
However, for a tensor $\mathcal{Y}$, its best rank-$(r_1, r_2,...,r_N)$
approximation is hard to be obtained in general. Thus, here we use another form to replace the exact best $(r_1, r_2,...,r_N)$-rank approximation. Our method is first to compute the
best rank-$r_i$ approximation of ${\mathcal{X}}^k_{<i>}$ for each $i$, then update $\mathcal{X}^{k+1}$ by the convex combination of the
refoldings of these rank-$r_i$ matrices, i.e.,
\begin{eqnarray}\label{scheme1}
{\mathcal{X}^{k+1}}=\sum_{i=1}^N w_i\mathscr{B}_i^*(R_{r_i}(\mathscr{B}_i({\mathcal{X}^{k}}-\tau
{\mathscr{A}}^*({\mathscr{A}}({\mathcal{X}}^{k})-\bm{b})))),
\end{eqnarray}
where ${\mathscr{B}}_i({\mathcal{X}})$ denotes the mode-$i$ unfolding of a tensor ${\mathcal{X}}\in {\mathbb{T}}$
for any $i\in \{1,2,\ldots,N\}$:
\begin{eqnarray*}
{\mathscr{B}}_i: {\mathbb{T}}\rightarrow {\mathbb{R}}^{n_i\times J_i}~~ \mathrm{with}~~ {\mathscr{B}}_i({\mathcal{X}}):={\mathcal{X}}_{<i>},
\end{eqnarray*}
and $\mathscr{B}_i^*$ denotes the adjoint operator of $\mathscr{B}_i$.

Now, we are ready to present the iterative hard thresholding algorithm for solving (\ref{p1}) as below and its convergence analysis
will be presented in the next section.
\begin{table}[H]
\centering  \tabcolsep 15pt
\begin{tabular}{l}
\hline
\textbf{Algorithm 3.1} ~Iterative hard thresholding for MnRA  \\
\hline
\textbf{Input}: $\mathscr{A},~\mathbf{b},~\mathcal{X}^0,~r_i,~\tau$  \vspace{0.5mm}\\
\qquad \textbf{while} not converged, \textbf{do}   \vspace{0.5mm}  \\
\qquad \qquad ${\mathcal{Y}}^k={\mathcal{X}}^k-\tau{\mathscr{A}}^*({\mathscr{A}}({\mathcal{X}}^k)-\bm{b})$ \vspace{0.5mm} \\
\qquad \qquad \textbf{for} $i = 1 : N$  \\
\qquad \qquad \qquad $M_i^k=R_{r_i}(\mathscr{B}_i({\mathcal{Y}}^k))$   \\
\qquad \qquad \textbf{end}  \\
\qquad \qquad ${\mathcal{X}}^{k+1}=\sum_{i=1}^N w_i\mathscr{B}_i^*(M_i^k)$  \\
\qquad  \textbf{end while} \vspace{0.5mm} \\
\textbf{Output}: $\mathcal{X}$ \\
\hline
\end{tabular}
\end{table}

\section{Convergence Results}

In this section, we concentrate on the convergence of Algorithm 3.1. Let ${\bm b}={\mathscr{A}}({\mathcal{X}}^*)$ with
${\mathcal{X}}^*\in \mathbb{T}$ being the true data to be restored, and it is known that ${\mathcal{X}}^*$ is an $(r_1,r_2,...,r_N)$-rank tensor, i.e., the rank of the mode-$i$ unfolding of ${\mathcal{X}}^*$ is not greater than $r_i$ for each $i\in \{1,\ldots,N\}$. Algorithm 3.1 is used to recover the true data ${\mathcal{X}}^*$. Next, we will prove the following inequality to characterize the performance of the proposed algorithm:
\begin{eqnarray*}
\|{\mathcal{X}}^*-{\mathcal{X}}^{k+1}\|_F\leq \alpha\|{\mathcal{X}}^*-\mathcal{X}^{k}\|_F,
\end{eqnarray*}
where ${\mathcal{X}}^k$ denotes the $k$-th iteration generated by Algorithm 3.1 and $\alpha \in (0,1)$ denotes the rate
at which the sequence converges to ${\mathcal{X}}^*$. The analysis begins by giving the following
concepts, including the restricted isometry constant (RIC) of a linear transformation and singular value decomposition (SVD) basis of a matrix.
\begin{definition}[Definition 1 in \citet{shzl2013}]\label{Df-t-ric}
Let ${\bf r}=(r_1,\ldots,r_N)$. The restricted isometry constant $\delta_{\bf r}$ of a linear transformation ${\mathscr {A}}: \mathbb{T}$
$\rightarrow \mathbb{R}^p$ with order ${\bf r}$ is defined as the smallest constant such that
\begin{eqnarray}\label{rip}
(1-\delta_{\bf r})\|\mathcal{X}\|_F^2\leq \|{\mathscr{A}}(\mathcal{X})\|_2^2\leq
(1+\delta_{\bf r})\|\mathcal{X}\|_F^2
\end{eqnarray}
holds for any $(r_1,\ldots,r_N)$-rank tensor $\mathcal{X}$, i.e., rank$(\mathcal{X}_{<i>})\leq r_i$ for all $i\in \{1,\ldots,N\}$.
\end{definition}

\begin{definition}[Definition 2.5 in \citet{gm10}]
Assume that the $r$-rank matrix $X_r$ has the SVD $X_r=\sum_{i=1}^r \sigma_i {\bm u}_i {\bm v}_i^{\top}$. $\Gamma:=
\{{\bm u}_1{\bm v}_1^{\top}, {\bm u}_2{\bm v}_2^{\top},\cdots,{\bm u}_r{\bm v}_r^{\top}\}$ is called an SVD basis for the matrix $X_r$.
\end{definition}
It's easy to see that the elements in the subspace spanned by the SVD basis are all $r$-rank matrices.
Based on these definitions, we give the following important lemma, which paves the way towards the convergence
of Algorithm 3.1.

\begin{lemma}[Lemma 4.1 in \citet{gm10}]\label{lemma-con}
Suppose $X:=R_r(Y)$ is the best $r$-rank approximation to the matrix $Y$ and $\Gamma$ is an SVD basis
of $X$. Then, for any $r$-rank matrix $X_r$ and SVD basis $\Gamma_r$ of $X_r$, we have
\begin{eqnarray*}
\|\mathcal{P}_{\mathbb{H}} (X)-\mathcal{P}_{\mathbb{H}} (Y)\|_F\leq \|\mathcal{P}_{\mathbb{H}} (X_r)-
\mathcal{P}_{\mathbb{H}} (Y)\|_F,
\end{eqnarray*}
where $\mathbb{H}$ is any orthonormal set of matrices satisfying $\mbox{\rm span}(\Gamma\cup\Gamma_r)
\subseteq\mbox{\rm span}(\mathbb{H})$, and $\mathcal{P}_{\mathbb{H}}(X)$ is the projection of $X$ onto the subspace spanned by
$\mathbb{H}$..
\end{lemma}

Now we prove the convergence of Algorithm 3.1 under proper conditions.
\begin{theorem}\label{convergence1}
Let ${\mathcal{X}}^*\in \mathbb{T}$ be the original data to be restored with ${\bm b}={\mathscr{A}}({\mathcal{X}}^*)$, and ${\mathcal{X}}^*$ is an $(r_1,r_2,...,r_N)$-rank tensor.
Set $J_i=\lceil\frac{n_i}{r_i}\rceil$ and $J=\max\limits_{1\leq i\leq n} J_i$, where $\lceil\cdot\rceil$ means rounding up.
Suppose that $\frac{3}{4}<\tau< \frac{5}{4}$, and let $\delta_{{\bf 3r}^i}$ be the RIC of ${\mathscr{A}}$ with order ${\bf 3r}^i=(n_1,\ldots,n_{i-1},3r_i,n_{i+1},\ldots,n_N)$
and $\delta=\max \limits_{1\leq i\leq N} \delta_{{\bf 3r}^i}$. If $\delta<\frac{\frac{1}{4}-|1-\tau|}{\tau(1+\lceil\log_2 J\rceil)}$, then the iterative sequence $\{{\mathcal{X}}^{k}\}$ generated by Algorithm 3.1 is linearly convergent to the original data ${\mathcal{X}}^*$ with rate $\frac{1}{2}$, i.e.,
\begin{eqnarray}\label{thm-0}
\|{\mathcal{X}}^*-{\mathcal{X}}^{k+1}\|_F\leq \frac{1}{2}\|{\mathcal{X}}^*-\mathcal{X}^{k}\|_F.
\end{eqnarray}
Moreover, iterating the above inequality, we have
\begin{eqnarray*}
\|{\mathcal{X}}^*-{\mathcal{X}}^k\|_F\leq 2^{-k}\|{\mathcal{X}}^*-\mathcal{X}^0\|_F.
\end{eqnarray*}
\end{theorem}

\begin{proof}
To facilitate, we denote $X^*_i:=\mathscr{B}_i(\mathcal{X}^*)$ and $X_i^k:=\mathscr{B}_i({\mathcal{X}}^k)$ for all  $i\in\{1,2...,N\}$ in the proof.
Since ${\mathcal{X}}^*\in \mathbb{T}$ is a $(r_1,r_2,...,r_N)$-rank tensor, we have $X^*_{i}$ is an $r_i$-rank matrix, i.e., the rank of $X^*_{i}$ is not greater than $r_i$. Note that from Algorithm 3.1, $M_i^k=R_{r_i}(\mathscr{B}_i({\mathcal{Y}}^k))$ is also an $r_i$-rank matrix for all $i\in \{1,2,\cdots,N\}$. Thus, for each $i\in\{1,2...,N\}$, there exist the SVD basises for $X^*_{i}$ and $M_i^k$, denoted by $\Gamma_i^*$ and $\Gamma_i^k$,
respectively. And let $\mathbb{H}_i^k$ denote an orthonormal basis of the subspace span$(\Gamma_i^*\cup\Gamma_i^k)$. Then the subspace spanned by $\mathbb{H}_i^k$, containing $X^*_{i}$ and $M_i^k$, is a set of $2r_i$-rank matrices. Setting
$\mathcal{P}_{\mathbb{H}_i^k}(Z)$ to be the projection of $Z$ onto the subspace spanned by
$\mathbb{H}_i^k$. Then, $\mbox{\rm rank}(\mathcal{P}_{\mathbb{H}_i^k}(Z))\leq 2r_i$.

Based on the aforementioned notations and the iterative scheme of Algorithm 3.1, we have
\begin{eqnarray}\label{thm-1}
&&\|{\mathcal{X}}^*-{\mathcal{X}}^{k+1}\|_F=\|{\mathcal{X}}^*-\sum_{i=1}^N w_i \mathscr{B}_i^*(M_i^k)\|_F     \nonumber \\
&\leq& \sum_{i=1}^N w_i \|X^*_i-M_i^k\|_F= \sum_{i=1}^N w_i \|\mathcal{P}_{\mathbb{H}_i^k}(X^*_{i})-\mathcal{P}_{\mathbb{H}_i^k}(M_i^k)\|_F    \nonumber \\
&\leq& \sum_{i=1}^N w_i \left(\|\mathcal{P}_{\mathbb{H}_i^k}(X^*_{i})-\mathcal{P}_{\mathbb{H}_i^k}(\mathscr{B}_i
({\mathcal{Y}}^k))\|_F+\|\mathcal{P}_{\mathbb{H}_i^k}(\mathscr{B}_i({\mathcal{Y}}^k))-\mathcal{P}_{\mathbb{H}_i^k}
(M_i^k)\|_F\right)        \nonumber \\
&\leq& 2\sum_{i=1}^N w_i \|\mathcal{P}_{\mathbb{H}_i^k}(X^*_{i})-\mathcal{P}_{\mathbb{H}_i^k}(\mathscr{B}_i
({\mathcal{Y}}^k))\|_F,
\end{eqnarray}
where the first and third inequality follow from the triangle inequality, the second equality
follows from $X^*_{i}, M_i^k\in \mbox{\rm span}(\mathbb{H}_i^k)$, and the last
inequality follows from Lemma \ref{lemma-con}. Furthermore, for each index $i\in\{1,2...,N\}$, we have
\begin{eqnarray}\label{thm-2}
&&  \|\mathcal{P}_{\mathbb{H}_i^k}(X^*_{i})-\mathcal{P}_{\mathbb{H}_i^k}(\mathscr{B}_i({\mathcal{Y}}^k))\|_F   \nonumber \\
&=& \|\mathcal{P}_{\mathbb{H}_i^k}(X^*_{i})-\mathcal{P}_{\mathbb{H}_i^k}(X_i^k-\tau\mathscr{B}_i{\mathscr{A}}^*
    ({\mathscr{A}}({\mathcal{X}}^k)-{\bm b}))\|_F    \nonumber \\
&=& \|\mathcal{P}_{\mathbb{H}_i^k}(X^*_{i}-X_i^k)-\tau\mathcal{P}_{\mathbb{H}_i^k}\mathscr{B}_i{\mathscr{A}}^*
    {\mathscr{A}}\mathscr{B}_i^*(X^*_{i}-X_i^k)\|_F   \nonumber\\
&\leq& |1-\tau|\|\mathcal{P}_{\mathbb{H}_i^k}(X^*_{i}-X_i^k)\|_F
+\tau\|(I-\mathcal{P}_{\mathbb{H}_i^k}\mathscr{B}_i{\mathscr{A}}^*{\mathscr{A}}\mathscr{B}_i^*
    \mathcal{P}_{\mathbb{H}_i^k})(X^*_{i}-X_i^k)\|_F    \nonumber \\
&& +\tau\|\mathcal{P}_{\mathbb{H}_i^k}\mathscr{B}_i{\mathscr{A}}^*{\mathscr{A}}\mathscr{B}_i^*
   (I-\mathcal{P}_{\mathbb{H}_i^k})(X^*_{i}-X_i^k)\|_F.
\end{eqnarray}
Therefore, in order to prove (\ref{thm-0}), i.e.,
$\|{\mathcal{X}}^*-{\mathcal{X}}^{k+1}\|_F\leq \frac{1}{2}\|{\mathcal{X}}^*-\mathcal{X}^{k}\|_F$, we need to estimate the upper bounds of the three terms in the right-hand side of (\ref{thm-2}), respectively. The specifical analysis is as follows:
\begin{itemize}
\item[(a)] ({\bf Estimation on the upper bound of the term} $|1-\tau|\|\mathcal{P}_{\mathbb{H}_i^k}(X^*_{i}-X_i^k)\|_F$)

Utilizing the non-expansion of projection operator, it's simple to estimate an upper bound of term $\|\mathcal{P}_{\mathbb{H}_i^k}(X^*_{i}-X_i^k)\|_F$ in the right-hand side of (\ref{thm-2}). This is given by
\begin{eqnarray}\label{thm-3}
|1-\tau|\|\mathcal{P}_{\mathbb{H}_i^k}(X^*_{i}-X_i^k)\|_F\leq |1-\tau| \|X^*_{i}-X_i^k\|_F= |1-\tau| \|\mathcal{X}^*-\mathcal{X}^k\|_F.
\end{eqnarray}

\item[(b)] ({\bf Estimation on the upper bound of} $\tau\|(I-\mathcal{P}_{\mathbb{H}_i^k}\mathscr{B}_i{\mathscr{A}}^*{\mathscr{A}}\mathscr{B}_i^*
\mathcal{P}_{\mathbb{H}_i^k})(X^*_{i}-X_i^k)\|_F$)

Note that, for any matrix $Z\in\mbox{\rm span}({\mathbb{H}_i^k})$ with appropriate size, $\mbox{\rm rank}(\mathcal{P}_{\mathbb{H}_i^k}(Z))\leq 2r_i$. Thus, $\mathscr{B}_i^*\mathcal{P}_{\mathbb{H}_i^k}(Z)$ is a ${\bf 2r}^i$-rank tensor, where ${\bf 2r}^i=(n_1,\ldots,n_{i-1},2r_i,n_{i+1},\ldots,n_N)$. Then, based on (\ref{Df-t-ric}),
\begin{eqnarray*}
(1-\delta_{{\bf 2r}^i})\|\mathscr{B}_i^*\mathcal{P}_{\mathbb{H}_i^k}(Z)\|_F^2
\leq\|\mathscr{A}\mathscr{B}_i^*\mathcal{P}_{\mathbb{H}_i^k}(Z)\|_F^2
\leq(1+\delta_{{\bf 2r}^i})\|\mathscr{B}_i^*\mathcal{P}_{\mathbb{H}_i^k}(Z)\|_F^2,
\end{eqnarray*}
which implies that
\begin{eqnarray*}
(1-\delta_{{\bf 2r}^i})\|\mathcal{P}_{\mathbb{H}_i^k}(Z)\|_F\leq\|\mathcal{P}_{\mathbb{H}_i^k}(\mathscr{B}_i\mathscr{A}^*\mathscr{A}
\mathscr{B}_i^*\mathcal{P}_{\mathbb{H}_i^k}(Z))\|_F\leq(1+\delta_{{\bf 2r}^i})\|\mathcal{P}_{\mathbb{H}_i^k}(Z)\|_F.
\end{eqnarray*}
Therefore, we can obtain that
\begin{eqnarray}\label{thm-4}
\tau\|(I-\mathcal{P}_{\mathbb{H}_i^k}\mathscr{B}_i{\mathscr{A}}^*{\mathscr{A}}\mathscr{B}_i^*\mathcal{P}_{\mathbb{H}_i^k})
(X^*_{i}-X_i^k)\|_F \leq \tau\delta_{{\bf 2r}^i}\|X^*_{i}-X_i^k\|_F=\tau\delta_{{\bf 2r}^i}\|\mathcal{X}^*-\mathcal{X}^k\|_F.
\end{eqnarray}

\item[(c)] ({\bf Estimation on the upper bound of} $\tau\|\mathcal{P}_{\mathbb{H}_i^k}(\mathscr{B}_i{\mathscr{A}}^*{\mathscr{A}}\mathscr{B}_i^*
(I-\mathcal{P}_{\mathbb{H}_i^k})(X^*_{i}-X_i^k))\|_F$)

Let the SVD of $X^*_{i}-X_i^k$ be $X^*_{i}-X_i^k=U \mathrm{Diag}
(\sigma)V^{\top}$, where $\sigma=(\sigma_1,\cdots,\sigma_{m-r_i})^{\top}$ is the vector of the singular values of
$X^*_{i}-X_i^k$ with $\sigma_1\geq \cdots\geq \sigma_{m-r_i}\geq 0$. We decompose $\sigma$ into a sum of vectors
$\sigma_{T_l}\;(l=1,2,\cdots)$, where disjoint index sets $T_1\cup T_2\cup \cdots \cup T_J = \{1,2,\cdots,m-r_i\}$ and
the sparsity of each $T_l$ is $r_i$ (except possibly $T_J$). Then, $Z_{T_1}$ is the part of $X^*_{i}-X_i^k$ corresponding
to the $r_i$ largest singular values, $Z_{T_2}$ is the part corresponding to the next $r_i$ largest singular values, and
so on. Thus, we have $X^*_{i}-X_i^k=U(\sum_{l=1}^{J_i} Z_{T_l}^k)V^{\top}$, where $Z_{T_l}^k$ is an $r_i$-rank matrix, $\|\mathscr{B}_i({\mathcal{X}}^*-\mathcal{X}^k)\|_F^2=\sum_{l=1}^{J_i} \|Z_{T_l}^k\|_F^2$, and $J_i=\lceil\frac{n_i}{r_i}\rceil$.
Then, we have
\begin{eqnarray}\label{thm-5}
&&\tau\|\mathcal{P}_{\mathbb{H}_i^k}(\mathscr{B}_i{\mathscr{A}}^*{\mathscr{A}}\mathscr{B}_i^*(I-\mathcal{P}_{\mathbb{H}_i^k})
(X^*_{i}-X_i^k))\|_F\nonumber\\
&\leq&\tau\sum_{l=1}^{J_i}\|\mathcal{P}_{\mathbb{H}_i^k}(\mathscr{B}_i{\mathscr{A}}^*{\mathscr{A}}\mathscr{B}_i^*(I-\mathcal{P}_{\mathbb{H}_i^k})
(Z_{T_l}^k))\|_F      \nonumber   \\
&\leq& \tau\delta_{{\bf 3r}^i}\sum_{l=1}^{J_i} \|(I-\mathcal{P}_{\mathbb{H}_i^k})Z_{T_l}^k\|_F        \nonumber     \\
&\leq& \tau\delta_{{\bf 3r}^i}\lceil\log_2 {J_i}\rceil \|{\mathcal{X}}^*-\mathcal{X}^k\|_F,
\end{eqnarray}
where the first inequality follows from the triangle inequality, and the second inequality follows from the following facts:
\begin{eqnarray*}
&&\|\mathcal{P}_{\mathbb{H}_i^k}(\mathscr{B}_i{\mathscr{A}}^*{\mathscr{A}}\mathscr{B}_i^*(I-\mathcal{P}_{\mathbb{H}_i^k})
(Z_{T_l}^k))\|_F\nonumber\\
&=&\max_{\|W\|_F=1}\langle W,\mathcal{P}_{\mathbb{H}_i^k}(\mathscr{B}_i{\mathscr{A}}^*{\mathscr{A}}\mathscr{B}_i^*(I-\mathcal{P}_{\mathbb{H}_i^k})
(Z_{T_l}^k))\rangle      \nonumber   \\
&=& \max_{\|W\|_F=1}\langle {\mathscr{A}}\mathscr{B}_i^*\mathcal{P}_{\mathbb{H}_i^k}(W),{\mathscr{A}}\mathscr{B}_i^*(I-\mathcal{P}_{\mathbb{H}_i^k})
(Z_{T_l}^k))\rangle    \nonumber     \\
&\leq& \max_{\|W\|_F=1} \delta_{{\bf 3r}^i} \|W\|_F\|Z_{T_l}^k\|_F \nonumber     \\
&=& \delta_{{\bf 3r}^i} \|Z_{T_l}^k\|_F,
\end{eqnarray*}
\end{itemize}

Therefore, by utilizing the results of items (a), (b), and (c), i.e., by combining (\ref{thm-3}), (\ref{thm-4}) and (\ref{thm-5}),
we can further obtain that for each index $i\in\{1,2...,N\}$
\begin{eqnarray}\label{thm-6}
\|\mathcal{P}_{\mathbb{H}_i^k}(X^*_{i})-\mathcal{P}_{\mathbb{H}_i^k}(\mathscr{B}_i({\mathcal{Y}}^k))\|_F
\leq (|1-\tau|+\tau\delta_{{\bf 2r}^i}+\tau\delta_{{\bf 3r}^i}\lceil\log_2 {J_i}\rceil)\|{\mathcal{X}}^*-\mathcal{X}^k\|_F.
\end{eqnarray}
Then, let $J=\max \limits_{1\leq i\leq n} J_i$ and $\delta=\max \limits_{1\leq i\leq N}\delta_{{\bf 3r}^i}$. Substituting (\ref{thm-6}) into
(\ref{thm-1}), we have
\begin{eqnarray}\label{thm-7}
\|{\mathcal{X}}^*-{\mathcal{X}}^{k+1}\|_F
&\leq&2\sum_{i=1}^N w_i \|\mathcal{P}_{\mathbb{H}_i^k}(X^*_{i})-
\mathcal{P}_{\mathbb{H}_i^k}(\mathscr{B}_i({\mathcal{Y}}^k))\|_F      \nonumber\\
&\leq& 2(\tau\delta+\tau\delta\lceil\log_2 {J}\rceil+|1-\tau|)\|{\mathcal{X}}^*-\mathcal{X}^k\|_F,
\end{eqnarray}
where the second inequality follows from that fact that $\delta_{{\bf 2r}^i}\leq \delta_{{\bf 3r}^i}$ for all $i$.

By the assumption that $\delta<\frac{\frac{1}{4}-|1-\tau|}{\tau(1+\lceil\log_2 J\rceil)}$, we have
\begin{eqnarray}\label{thm-8}
\|{\mathcal{X}}^*-{\mathcal{X}}^{k+1}\|_F\leq \frac{1}{2}\|{\mathcal{X}}^*-\mathcal{X}^k\|_F.
\end{eqnarray}
Iterating this inequality, we obtain
\begin{eqnarray*}
\|{\mathcal{X}}^*-{\mathcal{X}}^k\|_F\leq 2^{-k}\|{\mathcal{X}}^*-\mathcal{X}^0\|_F.
\end{eqnarray*}

The proof is complete.
\end{proof}
\vspace{3mm}

\textbf{Remark:} Note that we use a parameter $\tau>0$ as the step-size in Algorithm 3.1. Actually, $\tau$ is scoped. In the conditions of Theorem \ref{convergence1}, we assume $\frac{3}{4}<\tau<\frac{5}{4}$ to ensure that $\frac{\frac{1}{4}-|1-\tau|}{\tau(1+\lceil\log_2 J\rceil)}>0$, since $\delta>0$. Actually, observing (\ref{thm-7}) given in the proof of Theorem \ref{convergence1}, in order to ensure the convergence of the iterative sequence, we only need $2(\tau\delta+\tau\delta\lceil\log_2 {J}\rceil+||1-\tau|)<1$, which implies that $\delta<\frac{\frac{1}{2}-|1-\tau|}{\tau(1+\lceil\log_2 J\rceil)}$. Thus, $\delta>0$ implies that $\frac{1}{2}<\tau<\frac{3}{2}$, which is enough to guarantee the convergence of Algorithm 3.1.
\vspace{3mm}

Note that Theorem \ref{convergence1} considers the exact recovery case. However, it is possible to apply Algorithm 3.1 to recover the data corrupted by noise. Next, we will give the recoverability result of Algorithm 3.1 for the noisy case.
\begin{theorem}\label{convergence2}
Let ${\mathcal{X}}^*\in \mathbb{T}$ be the original data to be restored with ${\bm b}={\mathscr{A}}({\mathcal{X}}^*)+\epsilon$, where $\epsilon\in \mathbb{R}^p$ denotes the noise, and ${\mathcal{X}}^*$ is an $(r_1,r_2,...,r_N)$-rank tensor.
Set $J_i=\lceil\frac{n_i}{r_i}\rceil$ and $J=\max\limits_{1\leq i\leq n} J_i$, where $\lceil\cdot\rceil$ means rounding up.
Suppose that $\frac{3}{4}<\tau< \frac{5}{4}$, and let $\delta_{{\bf 3r}^i}$ be the RIC of ${\mathscr{A}}$ with order ${\bf 3r}^i=(n_1,\ldots,n_{i-1},3r_i,n_{i+1},\ldots,n_N)$
and $\delta=\max \limits_{1\leq i\leq N} \delta_{{\bf 3r}^i}$. If $\delta<\frac{\frac{1}{4}-|1-\tau|}{\tau(1+\lceil\log_2 J\rceil)}$, then Algorithm 3.1 will recover an approximation ${\mathcal{X}}^{k}$ satisfying
\begin{eqnarray*}
\|{\mathcal{X}}^*-{\mathcal{X}}^k\|_F\leq 2^{-k}\|{\mathcal{X}}^*-\mathcal{X}^0\|_F+2C\|\epsilon\|_2,
\end{eqnarray*}
where $C=2\tau\sqrt{1+\frac{\frac{1}{4}-|1-\tau|}{\tau(1+\lceil\log_2 J\rceil)}}$ is a constant, which only depends on $\tau$, $r_i$ and $n_i$ ($i = 1,2,\cdots,N$).
\end{theorem}
\begin{proof}
The proof is similar to the one of Theorem \ref{convergence1}. The main difference is to add one term involving the noise $\epsilon$. First, we can also obtain (\ref{thm-1}), i.e.,
\begin{eqnarray}\label{thm-21}
\|{\mathcal{X}}^*-{\mathcal{X}}^{k+1}\|_F\leq 2\sum_{i=1}^N w_i \|\mathcal{P}_{\mathbb{H}_i^k}(X^*_{i})-
\mathcal{P}_{\mathbb{H}_i^k}(\mathscr{B}_i({\mathcal{Y}}^k))\|_F.
\end{eqnarray}
In the noisy case (${\bm b}={\mathscr{A}}({\mathcal{X}}^*)+\epsilon$), we have the following result with similar deduction to (\ref{thm-2}), which only adds one term about $\epsilon$:
\begin{eqnarray}\label{thm-22}
&&  \|\mathcal{P}_{\mathbb{H}_i^k}(X^*_{i})-\mathcal{P}_{\mathbb{H}_i^k}(\mathscr{B}_i({\mathcal{Y}}^k))\|_F   \nonumber \\
&\leq& |1-\tau|\|\mathcal{P}_{\mathbb{H}_i^k}(X^*_{i}-X_i^k)\|_F
+\tau\|(I-\mathcal{P}_{\mathbb{H}_i^k}\mathscr{B}_i{\mathscr{A}}^*{\mathscr{A}}\mathscr{B}_i^*
    \mathcal{P}_{\mathbb{H}_i^k})(X^*_{i}-X_i^k)\|_F    \nonumber \\
&& +\tau\|\mathcal{P}_{\mathbb{H}_i^k}\mathscr{B}_i{\mathscr{A}}^*{\mathscr{A}}\mathscr{B}_i^*
   (I-\mathcal{P}_{\mathbb{H}_i^k})(X^*_{i}-X_i^k)\|_F+\tau\|\mathcal{P}_{\mathbb{H}_i^k}\mathscr{B}_i{\mathscr{A}}^*\epsilon\|_F.
\end{eqnarray}
The number of the right-hand terms increases 1, but the estimation of the remaining three terms are the same with (a), (b), (c) in the proof
of Theorem \ref{convergence1}. Thus, we just need to estimate the additional term $\tau\|\mathcal{P}_{\mathbb{H}_i^k}\mathscr{B}_i{\mathscr{A}}^*\epsilon\|_F$.
\begin{eqnarray}\label{thm-23}
\|\mathcal{P}_{\mathbb{H}_i^k}\mathscr{B}_i{\mathscr{A}}^*\epsilon\|_F&=&\max_{\|W\|_F=1}\langle W,\mathcal{P}_{\mathbb{H}_i^k}\mathscr{B}_i{\mathscr{A}}^*\epsilon\rangle
=\max_{\|W\|_F=1}\langle {\mathscr{A}}\mathscr{B}_i^*\mathcal{P}_{\mathbb{H}_i^k}(W),\epsilon\rangle \nonumber \\
&\leq& \max_{\|W\|_F=1}\|{\mathscr{A}}\mathscr{B}_i^*\mathcal{P}_{\mathbb{H}_i^k}(W)\|_F\|\epsilon\|_2 \nonumber \\
&\leq& \max_{\|W\|_F=1}\sqrt{1+\delta_{{\bf 2r}^i}}\|\mathscr{B}_i^*\mathcal{P}_{\mathbb{H}_i^k}(W)\|_F\|\epsilon\|_2 \nonumber \\
&\leq& \max_{\|W\|_F=1}\sqrt{1+\delta_{{\bf 2r}^i}}\|W\|_F\|\epsilon\|_2
= \sqrt{1+\delta_{{\bf 2r}^i}}\|\epsilon\|_2,
\end{eqnarray}
where the first inequality follows from the Cauchy-Schwarz inequality, the second inequality follows Definition \ref{Df-t-ric} and the fact that $\mathscr{B}_i^*\mathcal{P}_{\mathbb{H}_i^k}(W)$ is
a ${\bf 2r}^i$ tensor, where ${\bf 2r}^i=(n_1,\ldots,n_{i-1},2r_i,n_{i+1},\ldots,n_N)$, and the third inequality follows that $\|\mathscr{B}_i^*\|=1$ and $\|\mathcal{P}_{\mathbb{H}_i^k}(W)\|_F\leq \|W\|_F$, where $\|\mathscr{B}_i^*\|$ denotes the operator norm of $\mathscr{B}_i^*$.

Then, by using (a), (b), (c) in the proof of Theorem \ref{convergence1} and (\ref{thm-23}), we have
\begin{eqnarray}\label{thm-24}
&&  \|\mathcal{P}_{\mathbb{H}_i^k}(X^*_{i})-\mathcal{P}_{\mathbb{H}_i^k}(\mathscr{B}_i({\mathcal{Y}}^k))\|_F   \nonumber \\
&\leq& (|1-\tau|+\tau\delta_{{\bf 2r}^i}+\tau\delta_{{\bf 3r}^i}\lceil\log_2 {J_i}\rceil)\|{\mathcal{X}}^*-\mathcal{X}^k\|_F+\tau\sqrt{1+\delta_{{\bf 2r}^i}}\|\epsilon\|_2.
\end{eqnarray}
Substituting (\ref{thm-24}) into (\ref{thm-21}) and setting $J=\max \limits_{1\leq i\leq n} J_i$, $\delta=\max \limits_{1\leq i\leq N}\delta_{{\bf 3r}^i}$, we can obtain
\begin{eqnarray}\label{thm-27}
\|{\mathcal{X}}^*-{\mathcal{X}}^{k+1}\|_F
\leq 2(\tau\delta+\tau\delta\lceil\log_2 {J}\rceil+|1-\tau|)\|{\mathcal{X}}^*-\mathcal{X}^k\|_F+2\tau\sqrt{1+\delta}\|\epsilon\|_2,
\end{eqnarray}
Then, by the assumption that $\delta<\frac{\frac{1}{4}-|1-\tau|}{\tau(1+\lceil\log_2 J\rceil)}$, we can obtain
\begin{eqnarray}\label{thm-28}
\|{\mathcal{X}}^*-{\mathcal{X}}^{k+1}\|_F\leq \frac{1}{2}\|{\mathcal{X}}^*-\mathcal{X}^k\|_F+C\|\epsilon\|_2,
\end{eqnarray}
where $C=2\tau\sqrt{1+\frac{\frac{1}{4}-|1-\tau|}{\tau(1+\lceil\log_2 J\rceil)}}$ is a constant which only depends on $\tau$, $r_i$ and $n_i$ ($i\in\{1,\ldots,N\}$).
Iterating this inequality, we have
\begin{eqnarray*}
\|{\mathcal{X}}^*-{\mathcal{X}}^k\|_F\leq 2^{-k}\|{\mathcal{X}}^*-\mathcal{X}^0\|_F+2C\|\epsilon\|_2.
\end{eqnarray*}

The proof is complete.
\end{proof}

\section{Implementation Details}

\noindent \textit{Problem settings}. The random low $n$-rank tensor completion problems without noise we
considered in our numerical experiments are generated as in \citet{gry11}, \citet{mqlj2011} and \citet{yhs12}. For creating
a tensor $\mathcal{M}\in\mathbb{R}^{n_1 \times \ldots \times n_N}$ with $n$-rank $(r_1,r_2,\cdots,r_N)$, we
first generate a core tensor $\mathcal{S}\in \mathbb{R}^{r_1\times\cdots\times r_N}$ with i.i.d. Gaussian
entries ($\sim\mathcal{N}(0,1)$). Then, we generate matrixes $U_{1},\cdots,U_{N}$, with $U_{i}\in\mathbb{R}^{n_i\times r_i}$
whose entries are i.i.d. from  $\mathcal{N}(0,1)$ and set
\begin{eqnarray*}
\mathcal{M}:=\mathcal{S}\times_1U_{1}\times_2\cdots\times_N U_{N}.
\end{eqnarray*}
With this construction, the $n$-rank of $\mathcal{M}$ equals $(r_1,r_2,\cdots,r_N)$ almost surely.

We also conduct numerical experiments on random low $n$-rank tensor completion problems with
noisy data. For the noisy random low $n$-rank tensor completion problems, the tensor
$\mathcal{M}\in\mathbb{R}^{n_1 \times \ldots \times n_N}$ is corrupted
by a noise tensor $\mathcal{E}\in\mathbb{R}^{n_1 \times \ldots \times n_N}$ with independent normally distributed entries.
Then, $\mathcal{M}$ is taken to be
\begin{eqnarray}\label{noise}
\mathcal{M}:=\bar{\mathcal{M}}+\sigma \mathcal{E}=\mathcal{S}\times_1U_{1}\times_2\cdots\times_N U_{N}+\sigma \mathcal{E},
\end{eqnarray}
where $\sigma$ is the noise level.

We use $sr$ to denote the sampling ratio, i.e., a percentage $sr$ of the entries to be known and choose the
support of the known entries uniformly at random among all supports of size $sr\left(\prod^N_{i=1}n_i\right)$.
The values and the locations of the known entries of $\mathcal{M}$ are used as input for the algorithms. \\

\noindent \textit{Predicting $n$-rank}. In practice, the $n$-rank of the optimal solution is usually unknown.
Thus, we need to estimate the $n$-rank appropriately during the iterations. Inspired by the work
\citep{gm10}, we propose a heuristic for determining $n$-rank $\bm{r}$. We start with
$\bm{r}:=(\lceil \frac{n_1}{2} \rceil, \lceil \frac{n_2}{2} \rceil, \cdots, \lceil \frac{n_N}{2}
\rceil)$. In the $k$-th iteration ($k\geq 2$), for each $i$, we first choose $r_i$ as the number of singular values of $\mathscr{B}_i(\mathcal{Y}^{k-1})$ which are greater than
$\xi \bar{\sigma}^{k-1}$, where $\bar{\sigma}^{k-1}$ is the largest singular value of
$\mathscr{B}_i(\mathcal{Y}^{k-1})$ and $\xi \in (0, 1)$ is a given tolerance. Since the given tolerance sometimes truncates too many singular values, we need to increase $r_i$ occasionally. Note that
from the iterative scheme (\ref{scheme1}), we have that ${\mathscr{A}}^*({\mathscr{A}}({\mathcal{X}}^*)-{\bm b})=0$
at the optimal point $\mathcal{X}^*$. Thus, we increase $r_i$ by 1 whenever the Frobenius norm of
${\mathscr{A}}^*({\mathscr{A}}({\mathcal{X}}^k)-{\bm b})$ increased. Our numerical experience indicates
the efficiency of this heuristic for determining $\bm{r}$.       \\

\noindent \textit{Singular value decomposition}. Computing singular value decomposition is the main
computational cost even if we use a state-of-the-art code PROPACK \citep{propack}, especially when
the rank is relatively large. Therefore, for random low $n$-rank tensor completion problems without noise,
we use the Monte Carlo algorithm LinearTimeSVD
developed by \citet{prm2006} to compute an approximate SVD, which was also used in
\citet{gm10}, \citet{sdl2009} and \citet{yhhh2013} to reduce the computational cost. This LinearTimeSVD algorithm returns
an approximation to
the largest $sv$ singular values and the corresponding left singular vectors of a matrix $A \in \mathbb{R}^{m \times n}$
in linear $\mathcal{O}(m+n)$ time. We outline it below.
\begin{table}[H]
\centering  \tabcolsep 14pt
\small{\begin{tabular}{l}
\hline
Linear Time Approximate SVD Algorithm \citep{gm10,prm2006,sdl2009}   \vspace{1mm} \\
\quad \textbf{Input:} $A \in \mathbb{R}^{m \times n}$, $c_s, sv \in \mathbb{Z}^{+}$ s.t. $1 \leq sv \leq c_s \leq n$,
                       $\{p_j\}^n_{j=1}$ s.t. $p_j \geq 0$, $\sum^n_{j=1}p_j=1$. \\
\quad \textbf{Output:} $H_k \in \mathbb{R}^{m \times sv}$ and $\sigma_t(C)$, $t = 1,\ldots,sv$. \\
\quad \quad For $t=1:c_s$  \\
\quad \quad \quad Pick $i_t \in \{1,\ldots,n\}$ with $Pr[i_t=\alpha]=p_{\alpha}$, $\alpha=1,\ldots,n$.  \\
\quad \quad \quad Set $C^{(t)}=A^{(i_t)}/\sqrt{c_s p_{i_t}}$.   \\
\quad \quad Compute $C^{\top}C$ and its SVD; say $C^{\top}C=\sum^{c_s}_{t=1}\sigma^2_t(C)y^t(y^t)^{\top}$.  \\
\quad \quad Compute $h^t=Cy^t/\sigma_t(C)$ for $t=1,\ldots,sv$. \\
\quad \quad Return $H_{sv}$, where $H_{sv}^{(t)}=h^t$, and $\sigma_t(C)$, $t=1,\ldots,sv$. \\
\hline
\end{tabular}}
\end{table}

Thus, the outputs $\sigma_t(C)$, $t=1,\ldots,sv$ are approximations to the largest $sv$ singular values and
$H_{sv}^{(t)}$, $t=1,\ldots,sv$ are approximations to the corresponding left singular vectors of the matrix $A$.
The parameter settings we used in LinearTimeSVD algorithm are similar to those in \citet{sdl2009}.
To balance the computational time and accuracy of SVD of $C^{\top}C$, we choose a suitable $c_s=\lceil \min(n_i, T_i)/2 \rceil$
with $T_i=\prod_{k=1,k\neq i}^N n_k$ for each mode-$i$. All $p_j$'s are set to $1/T_i$ for simplicity. For the predetermined
parameter $sv$, in the $k$-th iteration, we let $sv$ equal to the predetermined $r_i$ for each mode-$i$.

On the other hand, for random low $n$-rank tensor completion problems with noisy data, to guarantee the accuracy
of the solution, we will use the matlab command $[U,S,V] = \mathrm{svd}(X,'\mathrm{econ}')$ to compute full SVD in our algorithms
although it may cost more time than the LinearTimeSVD algorithm does.

\section{Numerical Experiments}
In this section, we apply Algorithm 3.1 to solve the low $n$-rank tensor completion problem (\ref{p-add1}). We use IHTr-LRTC to
denote the algorithm in which the $n$-rank is specified, and IHT-LRTC to denote the algorithm in which the
$n$-rank is determined by the heuristic described in Section 5. We test IHTr-LRTC and IHT-LRTC on both
simulated and real world data with the missing data, and compare them with the latest tensor completion
algorithms, including FP-LRTC \citep{yhs12}, TENSOR-HC \citep{mqlj2011}, ADM-TR(E) \citep{gry11} and
HoRPCA (Higher-order Robust Principal Component Analysis) \citep{gq2014}. The Tucker decomposition algorithm based on the idea
of alternating least squares from the \textit{N-way toolbox for Matlab} \citep{ab2000} is also included, for which we use the correct $n$-rank $(r_{1}, \cdots, r_{N})$ (``N-way-E") and the higher $n$-rank $(r_{1} +1, \cdots, r_{N} +1)$ (``N-way-IE"). All numerical experiments are run in Matlab 7.14 on a HP Z800 workstation with an Intel Xeon(R) 3.33GHz CPU and 48GB of RAM.

For random low $n$-rank tensor completion problems without noise, the relative error
\begin{eqnarray*}
\mathrm{rel.err}:=\frac{||\mathcal{X}_{\mathrm{sol}}-\mathcal{M}||_F}{||\mathcal{M}||_F}
\end{eqnarray*}
is used to estimate the closeness of $\mathcal{X}_{\mathrm{sol}}$ to $\mathcal{M}$, where $\mathcal{X}_{\mathrm{sol}}$
is the ``optimal" solution produced by the algorithms and $\mathcal{M}$ is the original tensor.

For random low $n$-rank tensor completion problems with noisy data, we follow \cite{mqlj2011} to measure
the performance by the normalized root mean square error (NRMSE) on the complementary set $\Omega^{c}$:
\begin{eqnarray*}
\mathrm{NRMSE}(\mathcal{X}^{\mathrm{sol}}, \bar{\mathcal{M}}):=\frac{||\mathcal{X}^{\mathrm{sol}}_{\Omega^{c}}-\bar{\mathcal{M}}_{\Omega^{c}}||_F}
{\left(\mathrm{max}(\bar{\mathcal{M}}_{\Omega^{c}})-\mathrm{min}(\bar{\mathcal{M}}_{\Omega^{c}})\right)\sqrt{|\Omega^{c}|}}
\end{eqnarray*}
where $\bar{\mathcal{M}}$ is as in (\ref{noise}) and $|\Omega^{c}|$ denotes the cardinality of $\Omega^{c}$.

The stopping criterion we used for IHTr-LRTC and IHT-LRTC in all our numerical experiments is as follows:
\begin{eqnarray*}
\frac{\| \mathcal{X}^{k+1}-\mathcal{X}^k \|_F}{\mathrm{max}\{1, \|\mathcal{X}^k\|_F\}} < \mathrm{Tol},
\end{eqnarray*}
where Tol is a moderately small number, since when $\mathcal{X}^k$ gets close to an optimal solution
$\mathcal{X}^{\mathrm{opt}}$, the distance between $\mathcal{X}^k$ and $\mathcal{X}^{k+1}$ should become very small.

In IHTr-LRTC and IHT-LRTC, we choose the initial iteration to be $\mathcal{X}^0=0$ and set $\mathrm{Tol}=10^{-8}$.
The weighted parameters $w_i$ are set to $\frac{1}{N}$ for simplicity. Additionally, the parameter $\xi$ in
predicting $n$-rank is set to $10^{-2}$ for noiseless cases and $0.3$ for noisy cases. In FP-LRTC, we set
$\mu_1=1$, $\tau=10$, $\theta_{\mu}=1-sr$, $\bar{\mu}=1\times10^{-8}$, $\varepsilon = 10^{-2}$. In TENSOR-HC,
we set the regularization parameters $\lambda_{i}$, $i \in \{1,2,\cdots,N\}$ to $1$ and $\tau$ to $10$. In ADM-TR(E),
the parameters are set to $c_{\beta}=5, c_{\lambda}=5, \beta =1, \lambda = N$. For HoRPCA, we follow \cite{gq2014}
to keep $\mu$ constant and set $\mu = 10\mathrm{std}(\mathrm{vec}(\mathscr{M}_{\Omega}))$. The regularization
parameter\footnote{This regularization parameter $\lambda$ is different from that in authors' paper. It is given by authors
in their Matlab code for tensor completion, which can be downloaded from https://sites.google.com/site/tonyqin/research.}
$\lambda = 10^8$. It is stopped when the maximum of the relative primal and dual residuals decreased to below $10^{-8}$.

In FIG.1, we first numerically compare the recovery results with different values of $\tau$ by testing IHTr-LRTC and IHT-LRTC
on random noiseless low $n$-rank tensor completion problems with the tensor of size $20 \times 20 \times 30 \times 30$
and $n$-rank $(4,4,4,4)$. The sampling ratio is set to 0.3 and 0.6, respectively. It's worth noting that
though the assumption $\frac{1}{2}<\tau<\frac{3}{2}$ is given for ensuring convergence by theoretical analysis,
we find that IHTr-LRTC and IHT-LRTC can be convergent with choosing $\tau$ in a more
broad interval, which can be seen in the figure ($\tau$ is chosen from $\tau=0.1$ to $\tau=1.5$).
Additionally, it is obvious that the larger $\tau$ becomes, the less time it costs to recover a tensor
with lower relative error. Therefore, considering these situations, we can choose a larger $\tau$ to guarantee the low
error and less iterations. Specifically, we will set $\tau = 1.4$ for the remaining tests in this paper.

\begin{figure}[H]
\centering
\subfigure[]{\includegraphics[height=5.4cm]{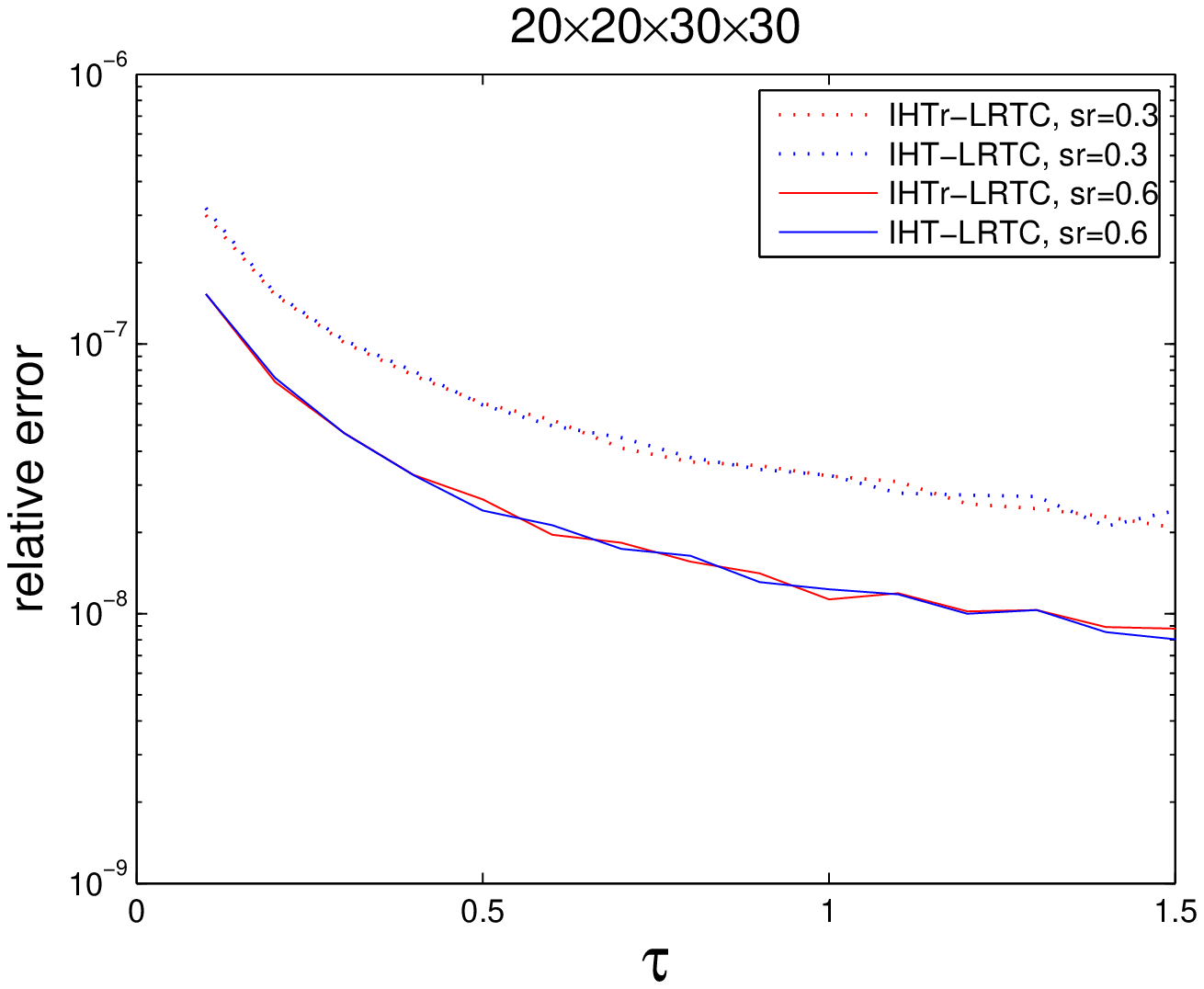}}
\subfigure[]{\includegraphics[height=5.4cm]{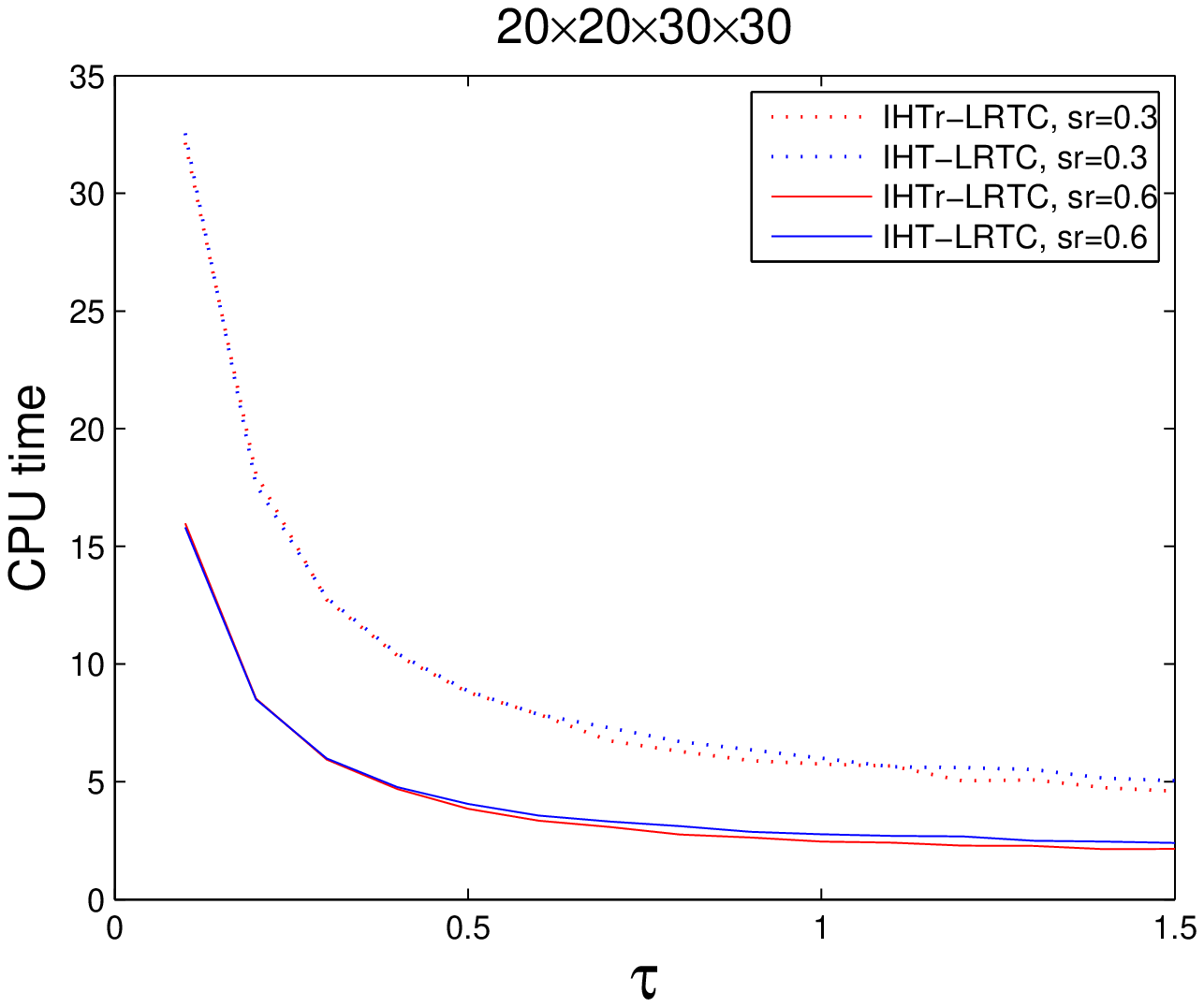}}
\vspace{-2mm}
\caption{Recovery results with different values of $\tau$ by testing IHTr-LRTC and IHT-LRTC
on random noiseless low $n$-rank tensor completion problems with the tensor of size $20 \times 20 \times 30 \times 30$
and $n$-rank $(4,4,4,4)$. (a) relative error; (b) CPU time in seconds.
All the results are average values of 10 independent trials.}
\end{figure}

Then, we compare IHTr-LRTC with IHT-LRTC on random noiseless low $n$-rank tensor completion problems
with the tensor of size $50 \times 50 \times 50$ and $n$-rank $(9, 9, 3)$.
The sampling ratio is set to 0.3 and 0.6, respectively. We plot the logarithm of the relative
error between the $\mathcal{X}^k$ and the true tensor $\mathcal{M}$ versus the iteration number for algorithms
IHTr-LRTC and IHT-LRTC in FIG.2 for each problem setting. From this figure, we can see that IHT-LRTC decreases
$\| \mathcal{X}^k-\mathcal{M}\|_F/\|\mathcal{M}\|_F$ slower than IHTr-LRTC due to the heuristic of determining
$r$. Additionally, for IHTr-LRTC, log$\|\mathcal{X}^k-\mathcal{M}\|_F/\|\mathcal{M}\|_F$ is approximately a linear
function of the iteration number $k$; for IHT-LRTC, it also approximately a linear function after several iterations.
This implies that the theoretical results in Theorem \ref{convergence1} approximately hold in practice.

\begin{figure}[H]
\centering
\includegraphics[height=6cm]{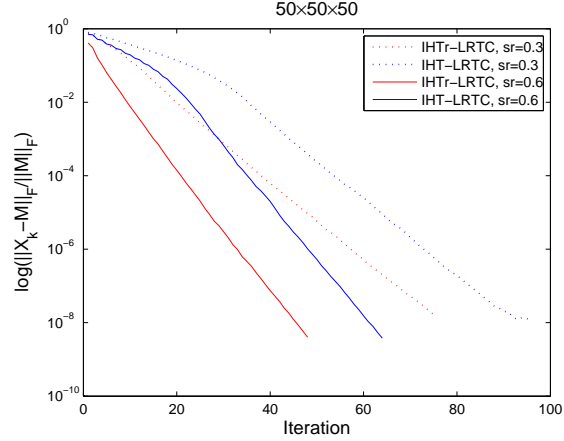}
\caption{Relative error versus the iteration number for algorithms IHTr-LRTC and IHT-LRTC on random noiseless
low $n$-rank tensor completion problems with the tensor of size $50 \times 50 \times 50$ and $n$-rank $(9, 9, 3)$.
The sampling ratio is set to 0.3 and 0.6, respectively.}
\end{figure}

\begin{table}[H]
\centering \tabcolsep 3pt
\begin{tabular}{ |l|llcl|l|llcl| }
\multicolumn{10}{c}{\small{\textbf{Table 1.} Comparisons of different algorithms on random noiseless low $n$-rank tensor completion
problems.}} \vspace{0.5mm}\\
\hline
\small{problem setting}  & \small{algorithm} & \small{iter} & \small{rel.err} & \footnotesize{time(s)} &
\small{problem setting}  & \small{algorithm} & \small{iter} & \small{rel.err} & \footnotesize{time(s)} \\
\hline
& \footnotesize{IHTr-LRTC} & \footnotesize{30} & \footnotesize{7.90e-9} & \footnotesize{0.12} &
& \footnotesize{IHTr-LRTC} & \footnotesize{82} & \footnotesize{2.33e-8} & \footnotesize{2.32} \vspace{-1mm} \\

& \footnotesize{IHT-LRTC} & \footnotesize{41} & \footnotesize{6.52e-9} & \footnotesize{0.17}  &
& \footnotesize{IHT-LRTC} & \footnotesize{93} & \footnotesize{2.38e-8} & \footnotesize{2.99}  \vspace{-1mm} \\

\small{$\mathbb{T}=\mathbb{R}^{20 \times 30 \times 40 }$}
& \footnotesize{FP-LRTC} & \footnotesize{105} & \footnotesize{1.92e-8} & \footnotesize{0.51} &
\small{$\mathbb{T}=\mathbb{R}^{60 \times 60 \times 60 }$}
& \footnotesize{FP-LRTC} & \footnotesize{520} & \footnotesize{2.00e-8} & \footnotesize{21.61}  \vspace{-1mm} \\

\small{${\bm r}=(2,2,2)$}
& \footnotesize{TENSOR-HC} & \footnotesize{66} & \footnotesize{2.09e-8} & \footnotesize{1.37} &
\small{${\bm r}=(9,9,6)$}
& \footnotesize{TENSOR-HC} & \footnotesize{49} & \footnotesize{7.23e-8} & \footnotesize{7.95} \vspace{-1mm} \\

\small{$sr=0.6$}
& \footnotesize{ADM-TR(E)} & \footnotesize{216} & \footnotesize{1.02e-8} & \footnotesize{10.40}  &
\small{$sr=0.3$}
& \footnotesize{ADM-TR(E)} & \footnotesize{410} & \footnotesize{2.37e-7} & \footnotesize{112.10} \vspace{-1mm} \\

& \footnotesize{HoRPCA} & \footnotesize{60} & \footnotesize{1.12e-8} & \footnotesize{1.53} &
& \footnotesize{HoRPCA} & \footnotesize{127} & \footnotesize{1.86e-8} & \footnotesize{18.09}   \vspace{-1mm} \\

& \footnotesize{N-way-E} & \footnotesize{31} & \footnotesize{1.40e-8} & \footnotesize{0.91} &
& \footnotesize{N-way-E} & \footnotesize{68} & \footnotesize{5.61e-8} & \footnotesize{8.25}    \vspace{-1mm} \\

& \footnotesize{N-way-IE} & \footnotesize{427} & \footnotesize{8.33e-2} & \footnotesize{14.38}   &
& \footnotesize{N-way-IE} & \footnotesize{772} & \footnotesize{2.54e-2} & \footnotesize{102.57}  \vspace{0mm} \\
\hline

& \footnotesize{IHTr-LRTC} & \footnotesize{31} & \footnotesize{7.62e-9} & \footnotesize{0.92} &
& \footnotesize{IHTr-LRTC} & \footnotesize{90} & \footnotesize{2.23e-8} & \footnotesize{4.90} \vspace{-1mm} \\

& \footnotesize{IHT-LRTC} & \footnotesize{38} & \footnotesize{7.20e-9} & \footnotesize{1.26}  &
& \footnotesize{IHT-LRTC} & \footnotesize{96} & \footnotesize{2.46e-8} & \footnotesize{5.28}  \vspace{-1mm} \\

\small{$\mathbb{T}=\mathbb{R}^{60 \times 60 \times 60 }$}
& \footnotesize{FP-LRTC} & \footnotesize{105} & \footnotesize{6.80e-9} & \footnotesize{4.49} &
\small{$\mathbb{T}=\mathbb{R}^{20 \times 20 \times 30 \times 30 }$}
& \footnotesize{FP-LRTC} & \footnotesize{520} & \footnotesize{3.67e-8} & \footnotesize{135.17}  \vspace{-1mm} \\

\small{${\bm r}=(9,9,6)$}
& \footnotesize{TENSOR-HC} & \footnotesize{35} & \footnotesize{3.33e-8} & \footnotesize{5.61} &
\small{${\bm r}=(4,4,4,4)$}
& \footnotesize{TENSOR-HC} & \footnotesize{50} & \footnotesize{3.42e-7} & \footnotesize{17.13} \vspace{-1mm} \\

\small{$sr=0.6$}
& \footnotesize{ADM-TR(E)} & \footnotesize{206} & \footnotesize{1.08e-8} & \footnotesize{60.62}  &
\small{$sr=0.3$}
& \footnotesize{ADM-TR(E)} & \footnotesize{456} & \footnotesize{2.56e-7} & \footnotesize{181.87} \vspace{-1mm} \\

& \footnotesize{HoRPCA} & \footnotesize{57} & \footnotesize{1.06e-8} & \footnotesize{8.22} &
& \footnotesize{HoRPCA} & \footnotesize{143} & \footnotesize{2.29e-8} & \footnotesize{34.11}   \vspace{-1mm} \\

& \footnotesize{N-way-E} & \footnotesize{26} & \footnotesize{9.97e-9} & \footnotesize{3.15} &
& \footnotesize{N-way-E} & \footnotesize{62} & \footnotesize{4.25e-8} & \footnotesize{27.75}    \vspace{-1mm} \\

& \footnotesize{N-way-IE} & \footnotesize{424} & \footnotesize{2.24e-2} & \footnotesize{55.42}   &
& \footnotesize{N-way-IE} & \footnotesize{804} & \footnotesize{3.63e-2} & \footnotesize{380.94}  \vspace{0mm} \\
\hline

& \footnotesize{IHTr-LRTC} & \footnotesize{37} & \footnotesize{9.62e-9} & \footnotesize{2.16} &
& \footnotesize{IHTr-LRTC} & \footnotesize{37} & \footnotesize{9.03e-9} & \footnotesize{21.55} \vspace{-1mm} \\

& \footnotesize{IHT-LRTC} & \footnotesize{45} & \footnotesize{8.10e-9} & \footnotesize{2.71}  &
& \footnotesize{IHT-LRTC} & \footnotesize{41} & \footnotesize{1.02e-8} & \footnotesize{24.40}  \vspace{-1mm} \\

\small{$\mathbb{T}=\mathbb{R}^{20 \times 20 \times 30 \times 30}$}
& \footnotesize{FP-LRTC} & \footnotesize{210} & \footnotesize{5.89e-9} & \footnotesize{16.69} &
\small{$\mathbb{T}=\mathbb{R}^{20 \times 20 \times 20 \times 20 \times 20 }$}
& \footnotesize{FP-LRTC} & \footnotesize{135} & \footnotesize{7.18e-9} & \footnotesize{103.83}  \vspace{-1mm} \\

\small{${\bm r}=(4,4,4,4)$}
& \footnotesize{TENSOR-HC} & \footnotesize{36} & \footnotesize{3.74e-8} & \footnotesize{12.32} &
\small{${\bm r}=(2,2,2,2,2)$}
& \footnotesize{TENSOR-HC} & \footnotesize{43} & \footnotesize{4.84e-8} & \footnotesize{198.59} \vspace{-1mm} \\

\small{$sr=0.6$}
& \footnotesize{ADM-TR(E)} & \footnotesize{219} & \footnotesize{1.88e-8} & \footnotesize{91.97}  &
\small{$sr=0.5$}
& \footnotesize{ADM-TR(E)} & \footnotesize{228} & \footnotesize{4.40e-8} & \footnotesize{728.66} \vspace{-1mm} \\

& \footnotesize{HoRPCA} & \footnotesize{65} & \footnotesize{1.38e-8} & \footnotesize{15.84} &
& \footnotesize{HoRPCA} & \footnotesize{64} & \footnotesize{1.30e-8} & \footnotesize{207.10}   \vspace{-1mm} \\

& \footnotesize{N-way-E} & \footnotesize{24} & \footnotesize{7.11e-9} & \footnotesize{11.09} &
& \footnotesize{N-way-E} & \footnotesize{29} & \footnotesize{5.30e-9} & \footnotesize{98.81}    \vspace{-1mm} \\

& \footnotesize{N-way-IE} & \footnotesize{442} & \footnotesize{2.12e-2} & \footnotesize{207.52}   &
& \footnotesize{N-way-IE} & \footnotesize{514} & \footnotesize{1.56e-2} & \footnotesize{1834.62}  \vspace{0mm} \\
\hline

\end{tabular}
\end{table}

Table 1 presents the different settings for random noiseless low $n$-rank tensor completion
problems and the recovery performance of different algorithms.
The order of the tensors varies from three to five, and we also vary the $n$-rank and the sampling ratio
$sr$. For each problem setting, we solve 10 randomly created tensor completion problems. iter, rel.err and
time(s) stands for the average iterations, the average relative error and the average time (seconds) for
each problem setting, respectively. From the results in Table 1, we can easily see that it costs
less time with lower $n$-rank and higher sampling ratio $sr$. By comparing the results of different algorithms,
it is easy to see that IHTr-LRTC and IHT-LRTC always perform better than other algorithms in both relative error
and CPU time. Note that though IHT-LRTC converges a little slower than IHTr-LRTC since it needs more iterations
and time to determine $n$-rank, the recoverability of IHT-LRTC can be comparable with that of IHTr-LRTC, which
indicates the efficiency of the heuristic for determining $n$-rank. For the problem with relative
large size (e.g., $\mathbb{T}=\mathbb{R}^{20 \times 20 \times 20 \times 20 \times 20}$, $\bm{r}=(2,2,2,2,2)$,
$sr=0.5$), we can see that IHTr-LRTC and IHT-LRTC can save much more time to recover a tensor.
Additionally, it's worth noting that N-way-E also has a good performance for all the problem settings, but N-way-IE
performs poorly for these problems though we just use a little higher $n$-rank. This situation indicates that
the \textit{N-way toolbox} depends strongly on the knowledge of the $n$-rank and the tensor may no longer
be recovered with the inexact $n$-rank.

Then, we test the first seven different algorithms (N-way-IE is poorer than other algorithms obviously
by Table 1) on random noiseless low $n$-rank tensor completion problems with the tensor of fixed size $100 \times 100 \times 100$
and different $n$-ranks $(r,r,r)$ (here we set $r_1=r_2=r_3=r$ for convenience). FIG.3 depict the average results of
10 independent trials corresponding to different $n$-rank $(r,r,r)$ for randomly created noiseless tensor completion problems.
The sampling ratios is set to $0.5$. As indicated in FIG.3, IHTr-LRTC and IHT-LRTC are
always faster and more robust than others, and provide the solutions with lower relative error.

\begin{figure}[H]
\centering
\subfigure[]{\includegraphics[height=5.4cm]{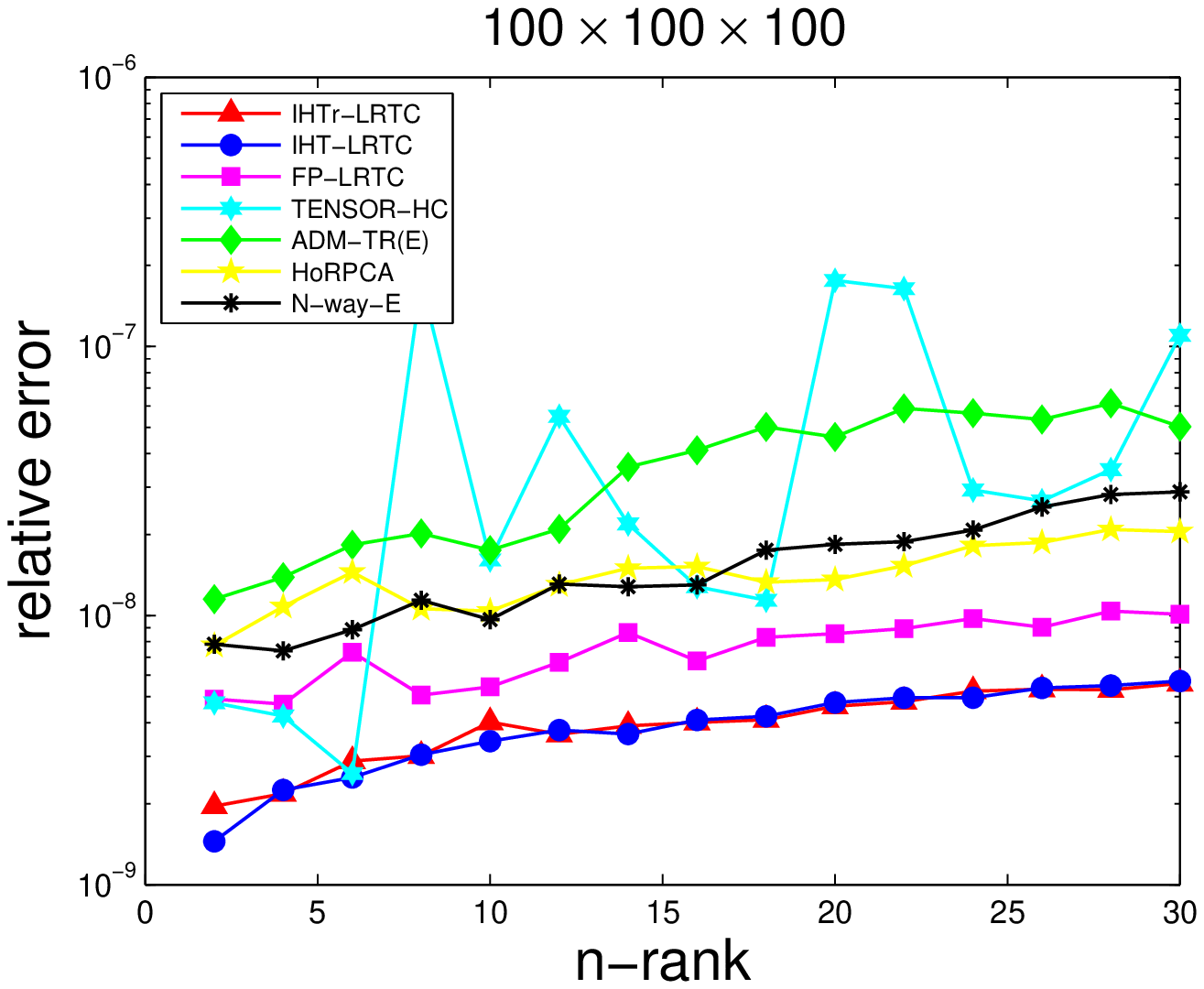}}
\subfigure[]{\includegraphics[height=5.4cm]{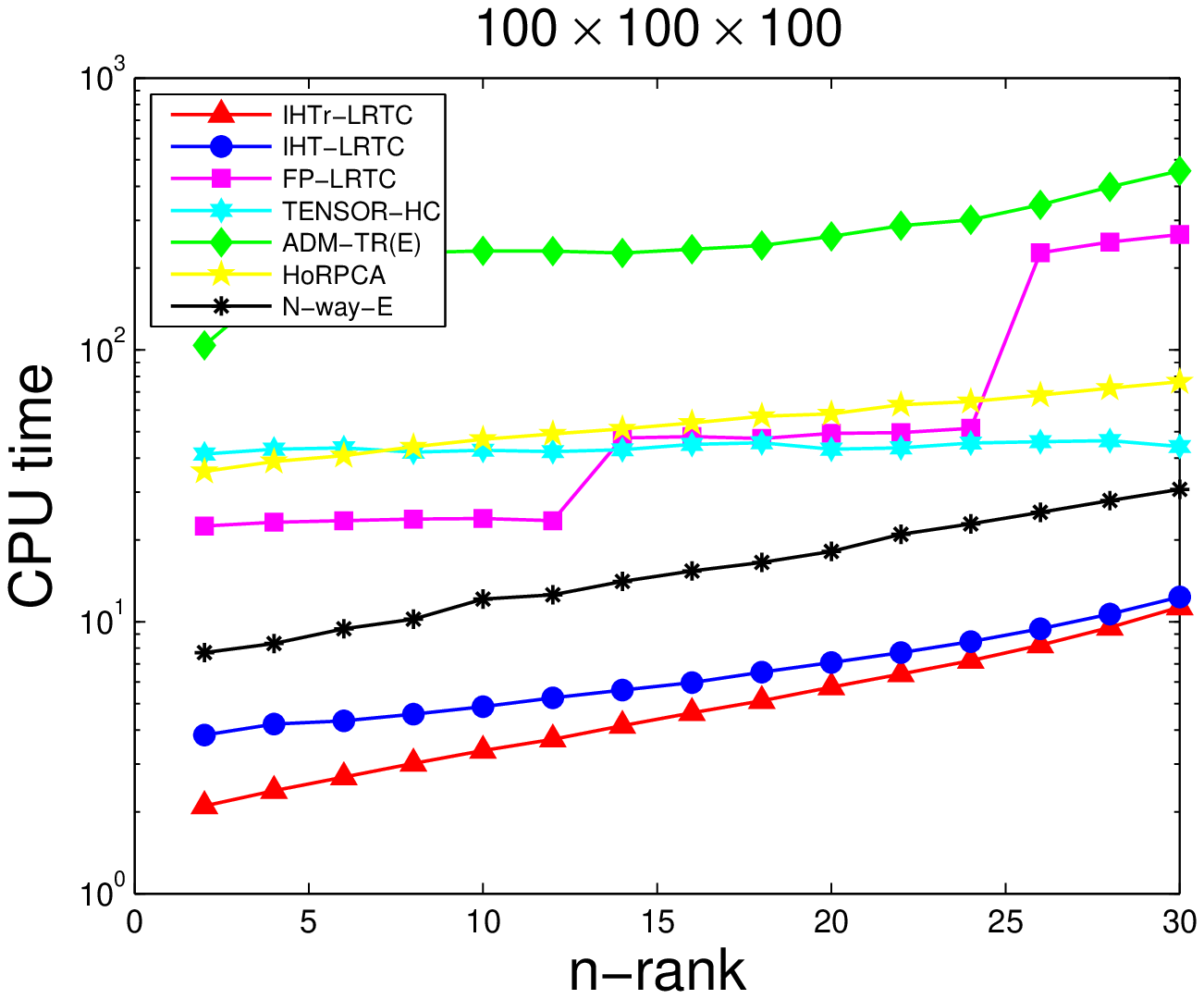}}
\caption{Recovery results by IHTr-LRTC, IHT-LRTC, FP-LRTC, TENSOR-HC, ADM-TR(E), HoRPCA and N-way-E on
random noiseless low $n$-rank tensor completion problems with the tensor of fixed size $100 \times 100 \times 100$
and different $n$-ranks. (a) relative error; (b) CPU time in seconds.
All the results are average values of 10 independent trials.}
\end{figure}

We further test the algorithms on random noisy low $n$-rank tensor completion problems. Table \textbf{2} presents
the numerical performance. In the table, we report the mean of NRMSEs, iterations and execution times over 10 independent
trials. Then, we set the noise level $\sigma=0.02$. From the results, we can easily see that IHTr-LRTC and IHT-LRTC are
comparable with other algorithms in terms of NRMSE and CPU time.

\begin{table}[H]
\centering \tabcolsep 3pt
\begin{tabular}{ |l|llcl|l|llcl| }
\multicolumn{10}{c}{\small{\textbf{Table 2.} Comparisons of different algorithms on random noisy low
$n$-rank tensor completion problems.}} \vspace{0.5mm}\\
\hline
\small{problem setting}  & \small{algorithm} & \small{iter} & \scriptsize{NRMSE} & \footnotesize{time(s)} &
\small{problem setting}  & \small{algorithm} & \small{iter} & \scriptsize{NRMSE} & \footnotesize{time(s)} \\
\hline
& \footnotesize{IHTr-LRTC} & \footnotesize{31} & \footnotesize{2.16e-3} & \footnotesize{0.44} &
& \footnotesize{IHTr-LRTC} & \footnotesize{74} & \footnotesize{2.77e-3} & \footnotesize{7.70} \vspace{-1mm} \\

& \footnotesize{IHT-LRTC} & \footnotesize{38} & \footnotesize{3.71e-3} & \footnotesize{0.55}  &
& \footnotesize{IHT-LRTC} & \footnotesize{102} & \footnotesize{3.06e-3} & \footnotesize{10.64}  \vspace{-1mm} \\

\small{$\mathbb{T}=\mathbb{R}^{20 \times 30 \times 40 }$}
& \footnotesize{FP-LRTC} & \footnotesize{105} & \footnotesize{5.33e-3} & \footnotesize{1.06} &
\small{$\mathbb{T}=\mathbb{R}^{60 \times 60 \times 60 }$}
& \footnotesize{FP-LRTC} & \footnotesize{520} & \footnotesize{1.19e-2} & \footnotesize{23.59}  \vspace{-1mm} \\

\small{${\bm r}=(2,2,2)$}
& \footnotesize{TENSOR-HC} & \footnotesize{45} & \footnotesize{9.22e-3} & \footnotesize{0.93} &
\small{${\bm r}=(9,9,6)$}
& \footnotesize{TENSOR-HC} & \footnotesize{31} & \footnotesize{8.88e-3} & \footnotesize{5.02} \vspace{-1mm} \\

\small{$sr=0.6$}
& \footnotesize{ADM-TR(E)} & \footnotesize{142} & \footnotesize{5.24e-3} & \footnotesize{9.30}  &
\small{$sr=0.3$}
& \footnotesize{ADM-TR(E)} & \footnotesize{301} & \footnotesize{1.20e-2} & \footnotesize{105.56} \vspace{-1mm} \\

\small{$\sigma=0.02$} & \footnotesize{HoRPCA} & \footnotesize{38} & \footnotesize{5.63e-3} & \footnotesize{0.97} &
\small{$\sigma=0.02$} & \footnotesize{HoRPCA} & \footnotesize{82} & \footnotesize{1.19e-2} & \footnotesize{12.05}   \vspace{-1mm} \\

& \footnotesize{N-way-E} & \footnotesize{32} & \footnotesize{1.24e-3} & \footnotesize{0.89} &
& \footnotesize{N-way-E} & \footnotesize{69} & \footnotesize{1.70e-3} & \footnotesize{8.23}    \vspace{-1mm} \\

& \footnotesize{N-way-IE} & \footnotesize{682} & \footnotesize{3.81e-3} & \footnotesize{19.29}   &
& \footnotesize{N-way-IE} & \footnotesize{748} & \footnotesize{2.03e-3} & \footnotesize{88.88}  \vspace{0mm} \\
\hline

& \footnotesize{IHTr-LRTC} & \footnotesize{30} & \footnotesize{2.89e-3} & \footnotesize{3.11} &
& \footnotesize{IHTr-LRTC} & \footnotesize{78} & \footnotesize{2.04e-3} & \footnotesize{13.50} \vspace{-1mm} \\

& \footnotesize{IHT-LRTC} & \footnotesize{39} & \footnotesize{3.22e-3} & \footnotesize{4.14}  &
& \footnotesize{IHT-LRTC} & \footnotesize{100} & \footnotesize{2.00e-3} & \footnotesize{17.35}  \vspace{-1mm} \\

\small{$\mathbb{T}=\mathbb{R}^{60 \times 60 \times 60 }$}
& \footnotesize{FP-LRTC} & \footnotesize{105} & \footnotesize{7.26e-3} & \footnotesize{5.00} &
\small{$\mathbb{T}=\mathbb{R}^{20 \times 20 \times 30 \times 30 }$}
& \footnotesize{FP-LRTC} & \footnotesize{520} & \footnotesize{1.45e-2} & \footnotesize{40.37}  \vspace{-1mm} \\

\small{${\bm r}=(9,9,6)$}
& \footnotesize{TENSOR-HC} & \footnotesize{23} & \footnotesize{9.64e-3} & \footnotesize{3.79} &
\small{${\bm r}=(4,4,4,4)$}
& \footnotesize{TENSOR-HC} & \footnotesize{26} & \footnotesize{9.74e-3} & \footnotesize{8.84} \vspace{-1mm} \\

\small{$sr=0.6$}
& \footnotesize{ADM-TR(E)} & \footnotesize{125} & \footnotesize{6.82e-3} & \footnotesize{46.05}  &
\small{$sr=0.3$}
& \footnotesize{ADM-TR(E)} & \footnotesize{530} & \footnotesize{1.47e-2} & \footnotesize{186.29} \vspace{-1mm} \\

\small{$\sigma=0.02$} & \footnotesize{HoRPCA} & \footnotesize{32} & \footnotesize{6.81e-2} & \footnotesize{4.69} &
\small{$\sigma=0.02$} & \footnotesize{HoRPCA} & \footnotesize{355} & \footnotesize{1.46e-2} & \footnotesize{87.08}   \vspace{-1mm} \\

& \footnotesize{N-way-E} & \footnotesize{26} & \footnotesize{1.34e-3} & \footnotesize{3.11} &
& \footnotesize{N-way-E} & \footnotesize{60} & \footnotesize{8.06e-4} & \footnotesize{26.04}    \vspace{-1mm} \\

& \footnotesize{N-way-IE} & \footnotesize{444} & \footnotesize{1.49e-3} & \footnotesize{51.43}   &
& \footnotesize{N-way-IE} & \footnotesize{925} & \footnotesize{1.21e-3} & \footnotesize{417.96}  \vspace{0mm} \\
\hline

& \footnotesize{IHTr-LRTC} & \footnotesize{35} & \footnotesize{2.68e-3} & \footnotesize{6.29} &
& \footnotesize{IHTr-LRTC} & \footnotesize{34} & \footnotesize{1.52e-3} & \footnotesize{83.79} \vspace{-1mm} \\

& \footnotesize{IHT-LRTC} & \footnotesize{42} & \footnotesize{2.26e-3} & \footnotesize{7.54}  &
& \footnotesize{IHT-LRTC} & \footnotesize{45} & \footnotesize{1.21e-3} & \footnotesize{111.25}  \vspace{-1mm} \\

\small{$\mathbb{T}=\mathbb{R}^{20 \times 20 \times 30 \times 30}$}
& \footnotesize{FP-LRTC} & \footnotesize{210} & \footnotesize{8.31e-3} & \footnotesize{17.46} &
\small{$\mathbb{T}=\mathbb{R}^{20 \times 20 \times 20 \times 20 \times 20 }$}
& \footnotesize{FP-LRTC} & \footnotesize{135} & \footnotesize{6.06e-3} & \footnotesize{117.55}  \vspace{-1mm} \\

\small{${\bm r}=(4,4,4,4)$}
& \footnotesize{TENSOR-HC} & \footnotesize{21} & \footnotesize{9.81e-3} & \footnotesize{7.04} &
\small{${\bm r}=(2,2,2,2,2)$}
& \footnotesize{TENSOR-HC} & \footnotesize{15} & \footnotesize{8.50e-3} & \footnotesize{73.24} \vspace{-1mm} \\

\small{$sr=0.6$}
& \footnotesize{ADM-TR(E)} & \footnotesize{204} & \footnotesize{8.01e-3} & \footnotesize{75.99}  &
\small{$sr=0.5$}
& \footnotesize{ADM-TR(E)} & \footnotesize{422} & \footnotesize{5.76e-3} & \footnotesize{1278.79} \vspace{-1mm} \\

\small{$\sigma=0.02$} & \footnotesize{HoRPCA} & \footnotesize{128} & \footnotesize{7.98e-3} & \footnotesize{32.32} &
\small{$\sigma=0.02$} & \footnotesize{HoRPCA} & \footnotesize{662} & \footnotesize{5.99e-3} & \footnotesize{2214.99}   \vspace{-1mm} \\

& \footnotesize{N-way-E} & \footnotesize{24} & \footnotesize{5.80e-4} & \footnotesize{11.17} &
& \footnotesize{N-way-E} & \footnotesize{28} & \footnotesize{1.20e-4} & \footnotesize{95.27}    \vspace{-1mm} \\

& \footnotesize{N-way-IE} & \footnotesize{450} & \footnotesize{9.20e-4} & \footnotesize{206.41}   &
& \footnotesize{N-way-IE} & \footnotesize{441} & \footnotesize{2.32e-4} & \footnotesize{1520.44}  \vspace{0mm} \\
\hline

\end{tabular}
\end{table}

\textit{Inpainting of color Images via low $n$-rank tensor completion}. Next, we further evaluate the performance
of IHTr-LRTC and IHT-LRTC on image inpainting \citep{bmgv2000}. Color images can be expressed as third-order
tensors. If the image is of low $n$-rank, or numerical low $n$-rank, we can solve the image inpainting
problem as a low $n$-rank tensor recovery problem. In our test, we first compute the best rank-$(r_1, r_2, r_3)$
approximation of a color image to obtain an numerical low $n$-rank image. Then, we randomly remove the values
of some of the pixels of the numerical low $n$-rank image, and want to fill in these missing values.
\vspace{2mm}

\textbf{Remark}: The best rank-$(r_1, r_2, \cdots, r_N)$ approximation is used as a tool for dimensionality
reduction and signal subspace estimation. Several algorithms for this purpose have been proposed in the literature,
e.g., the higher-order orthogonal iteration (HOOI) \citep{lbj2000}. More details can be seen in \citet{m2009}.
Note that the \textit{N-way toolbox for Matlab} is also an effective and convenient tool of computing
the best rank-$(r_1, r_2, \cdots, r_N)$ approximation. However, considering to be fair and reasonable,
we here use the \textit{MATLAB Tensor Toolbox} by \citet{bk2012}, which is an another famous tool for tensor computation,
to compute the best rank-$(r_1, r_2, \cdots, r_N)$ approximation. Using Matlab notation, for a tensor $\mathcal{X} \in
\mathbb{R}^{n_1 \times n_2 \times \cdots \times n_N}$, $\bar{\mathcal{X}}=\mathrm{tucker}\_\mathrm{als}
(\mathcal{X}, [r_1~ r_2 \cdots r_N])$ returns the best rank-$(r_1, r_2, \cdots, r_N)$ approximation of $\mathcal{X}$.
Additionally, the parameter $\xi$ in predicting $n$-rank is set to $10^{-4}$ to guarantee the better prediction
of $n$-rank for the practical problems.
\vspace{2mm}

FIG.4 and FIG.5 respectively present the recovered images for the best rank-$(30,30,3)$ and rank-$(100,100,3)$ approximation
of the original $512\times512$ image by different algorithms (Here, ADM-TR(E) and N-way-IE perform poorer than others obviously,
so their results are no longer reported). The sampling ratio is set to 0.3. We also report the numerical
results in Table 3. Although the recovered images of these five algorithms are similar visually to each other, the results in
Table 3 show that IHTr-LRTC and IHT-LRTC are more effective than others, especially for the problem with high $n$-rank.
More specifically, for the best rank-$(30,30,3)$ approximation of the original image, all the algorithms can recover the
image well by using only 30\% of pixels and IHTr-LRTC is much faster than others. For the best rank-$(100,100,3)$ approximation
of the original image, we can see that the relative errors of recovered images by FP-LRTC, TENSOR-HC and HoRPCA are very large
due to the relatively high $n$-rank. However, IHTr-LRTC and IHT-LRTC can also perform well.

\begin{figure}[H]
\centering
\subfigure[]{\includegraphics[height=4cm]{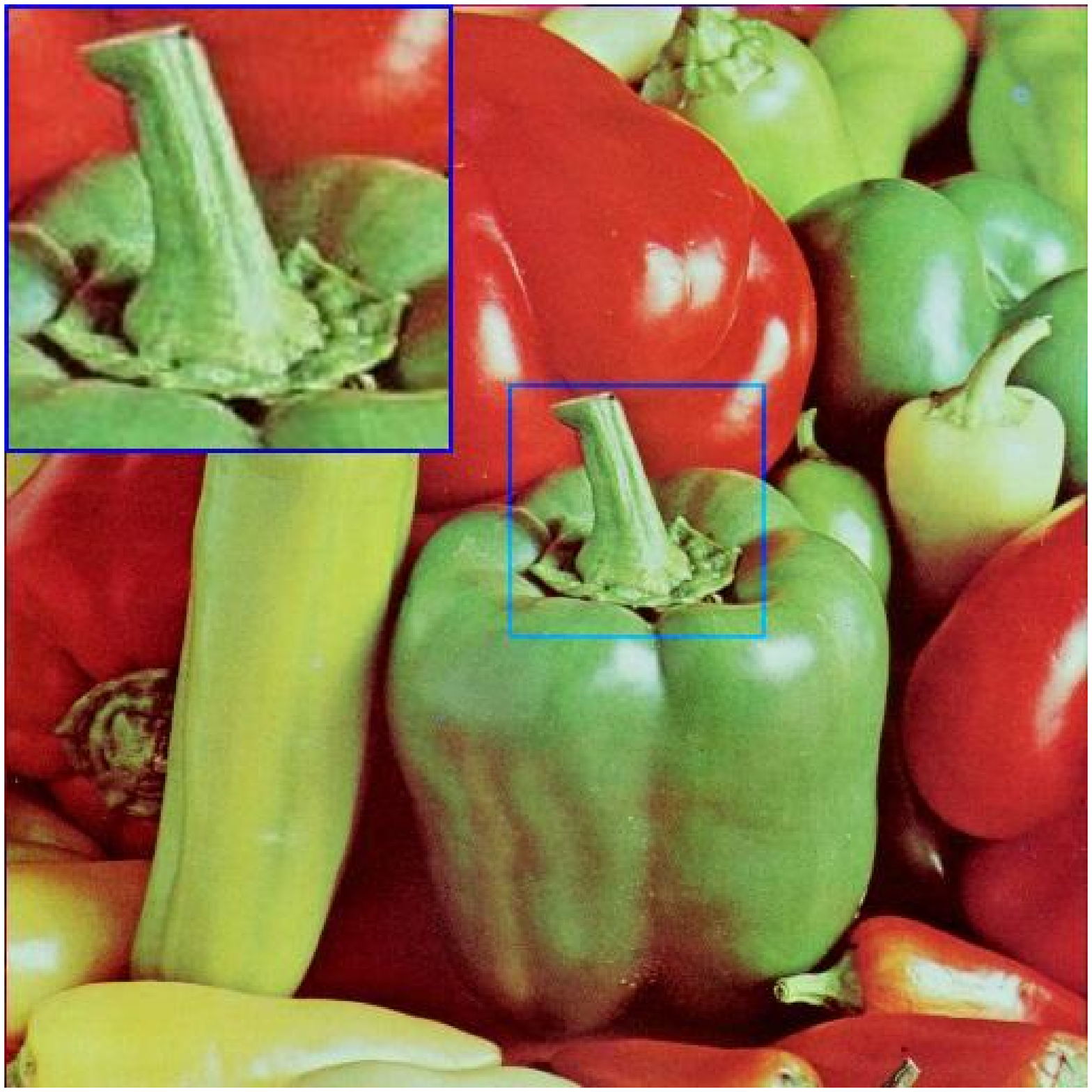}} \qquad
\subfigure[]{\includegraphics[height=4cm]{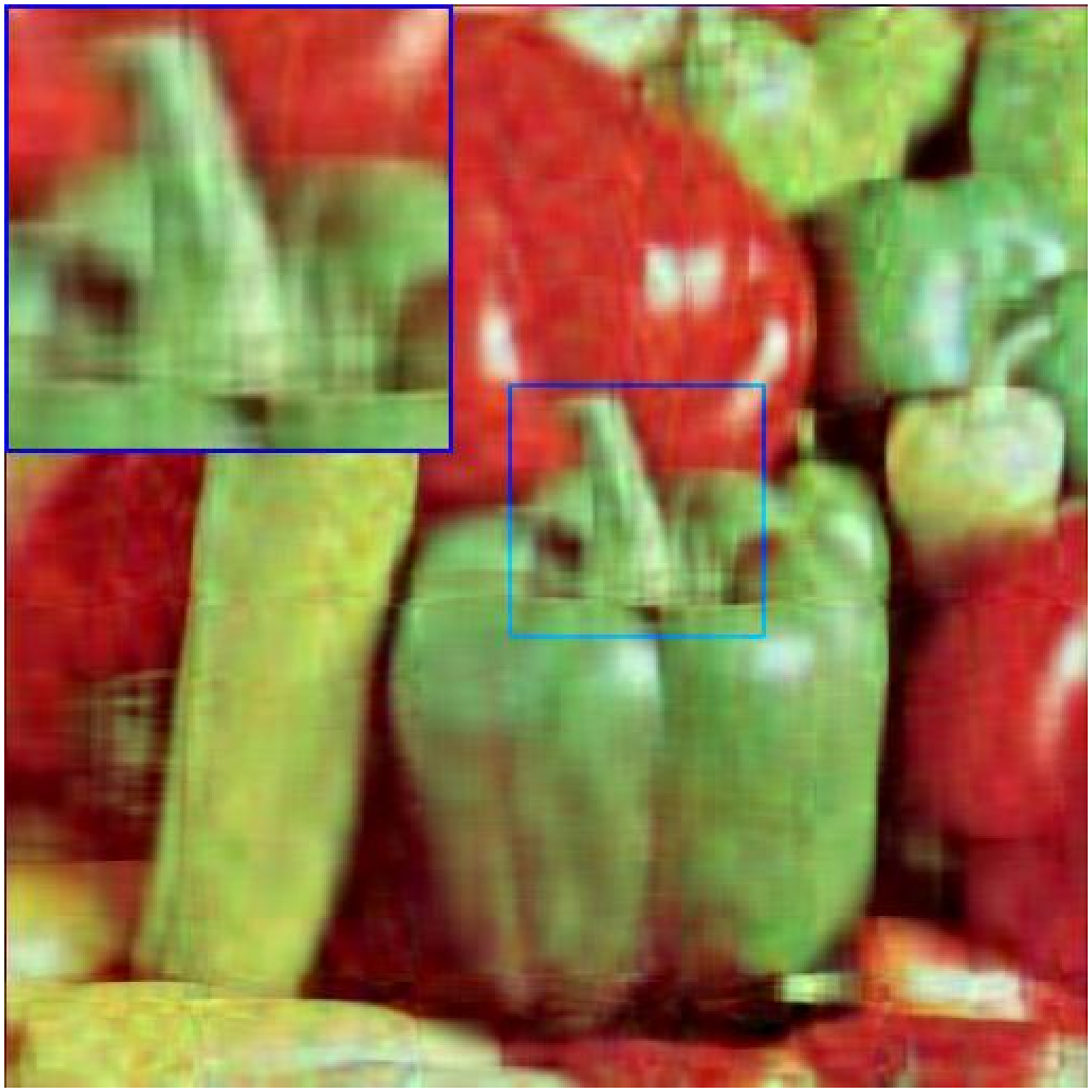}}~
\subfigure[]{\includegraphics[height=4cm]{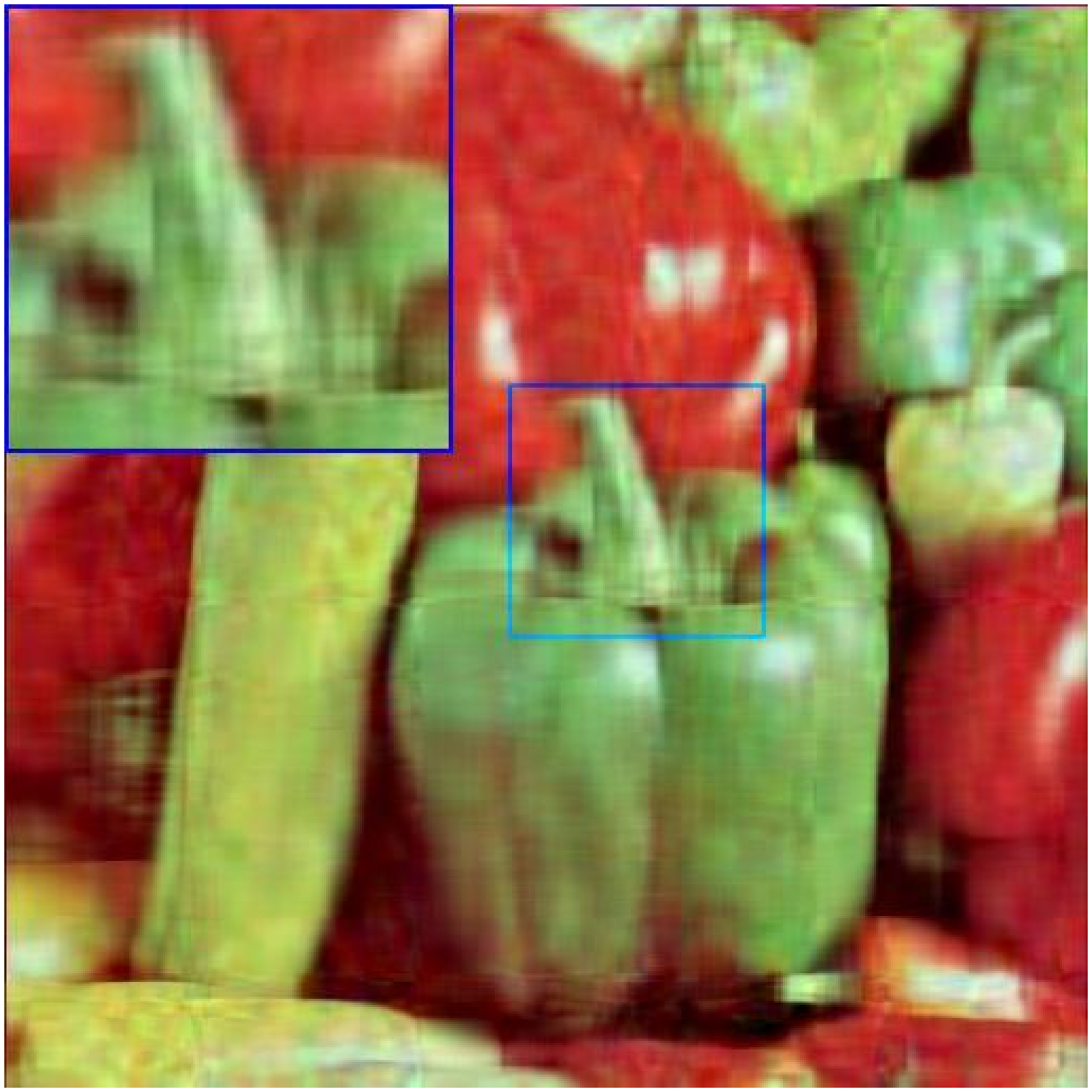}} \\

\subfigure[]{\includegraphics[height=4cm]{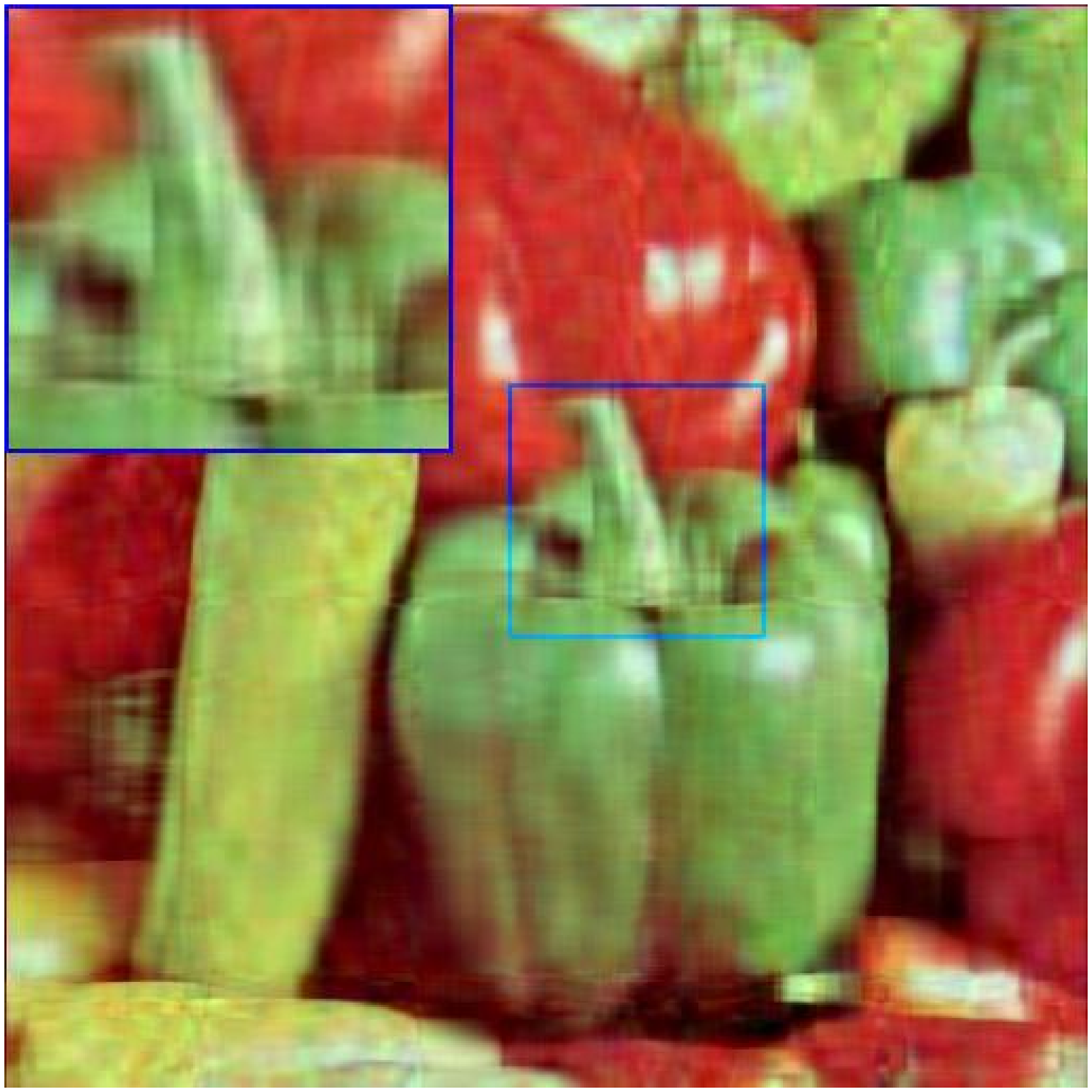}} \qquad
\subfigure[]{\includegraphics[height=4cm]{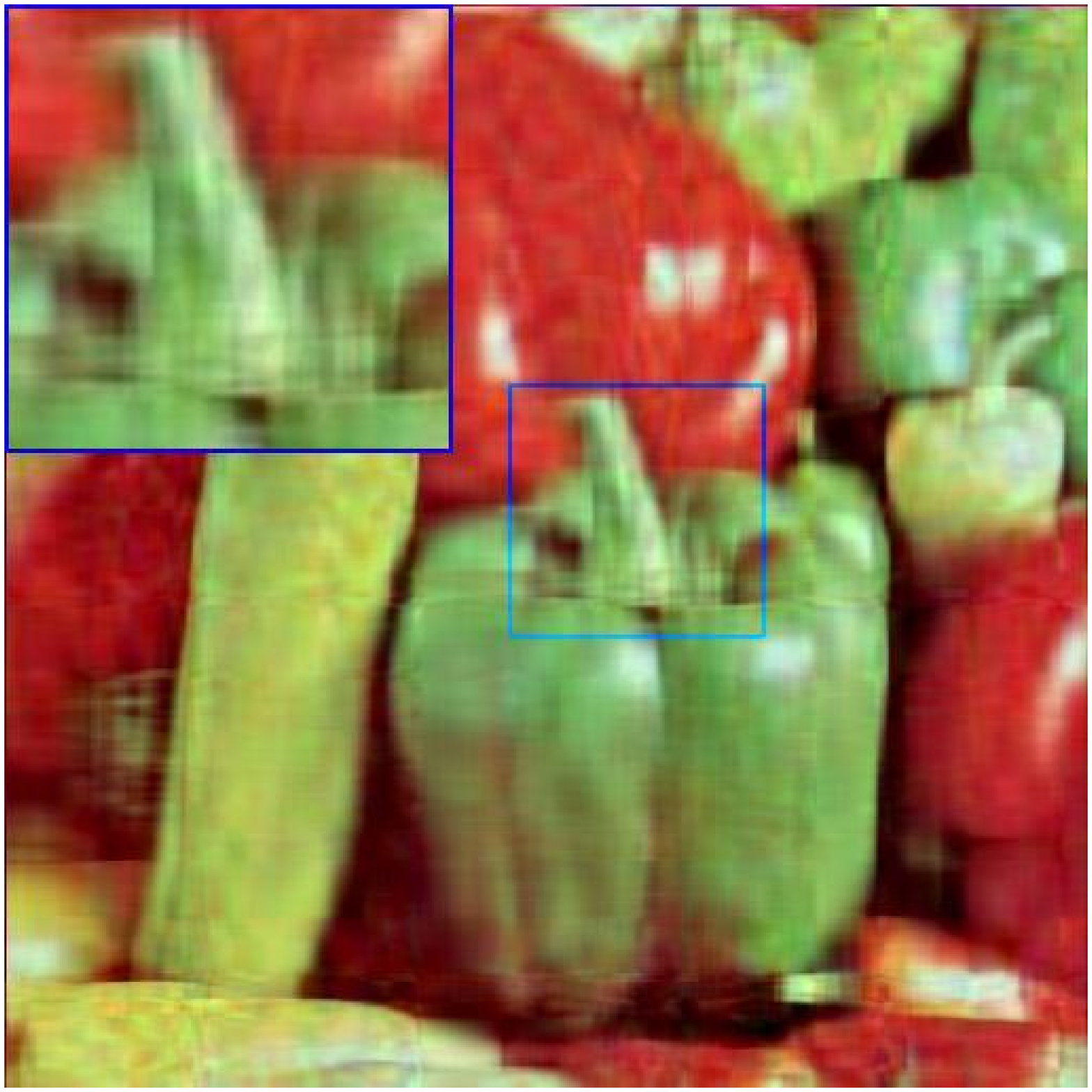}}~
\subfigure[]{\includegraphics[height=4cm]{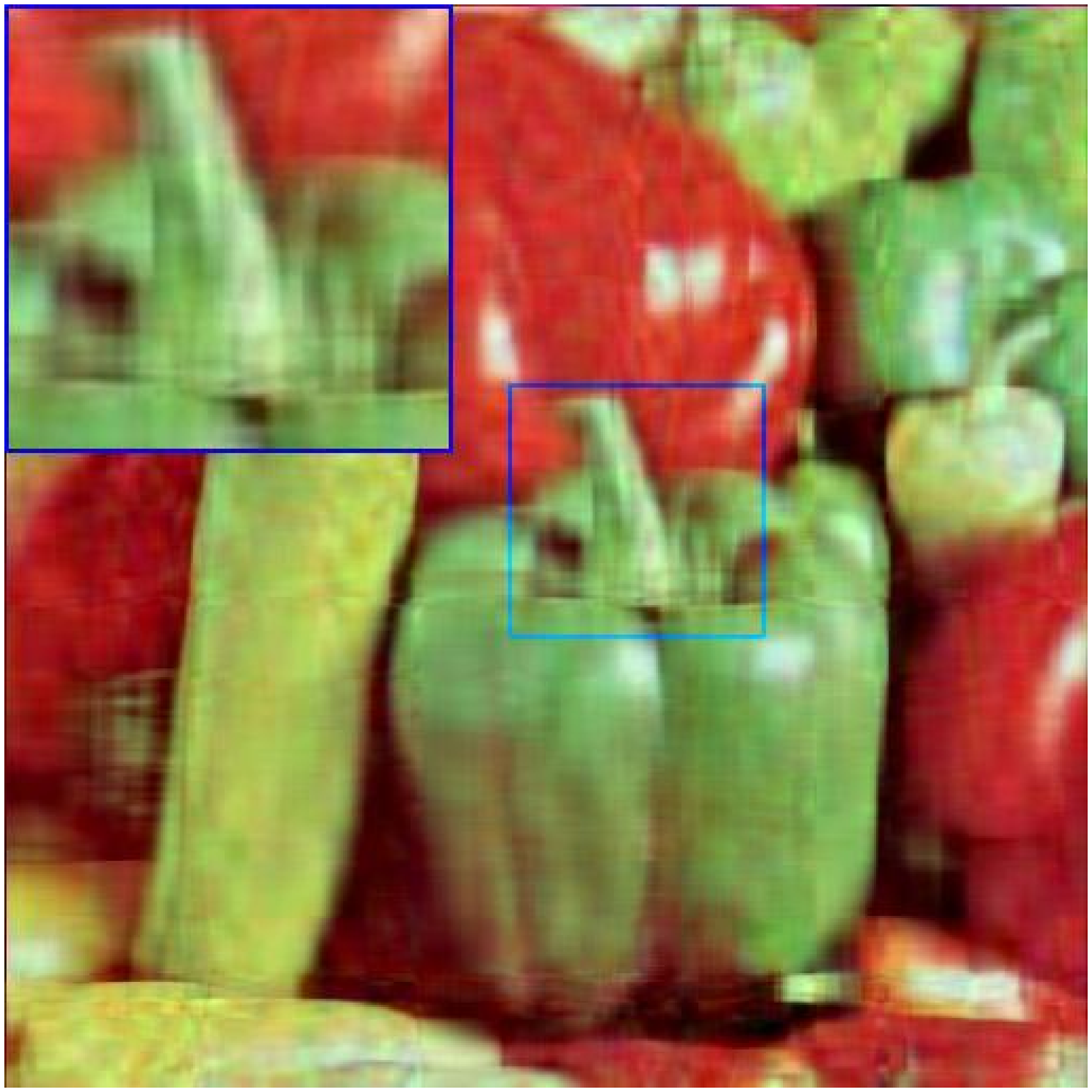}} \\

\subfigure[]{\includegraphics[height=4cm]{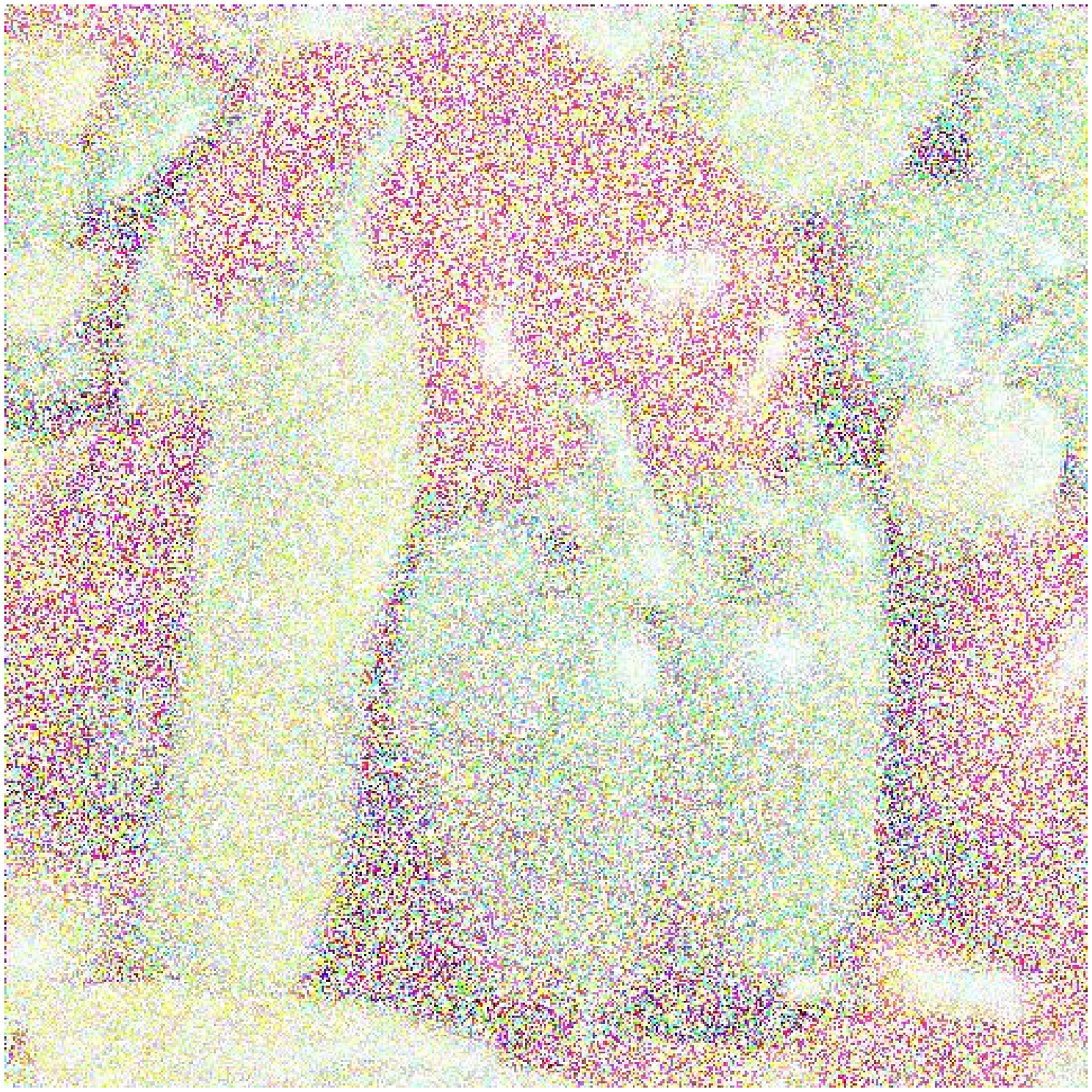}} \qquad
\subfigure[]{\includegraphics[height=4cm]{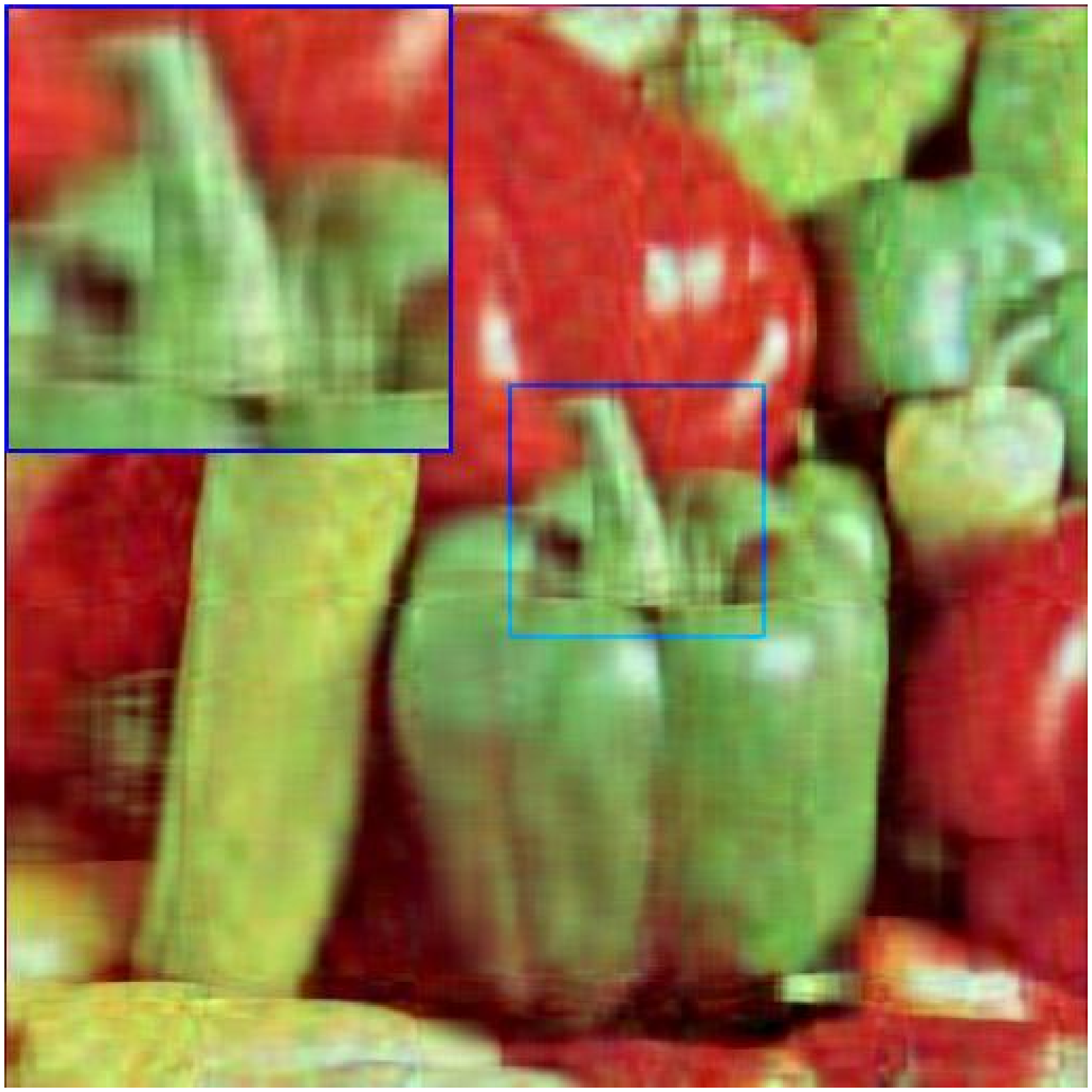}}~
\subfigure[]{\includegraphics[height=4cm]{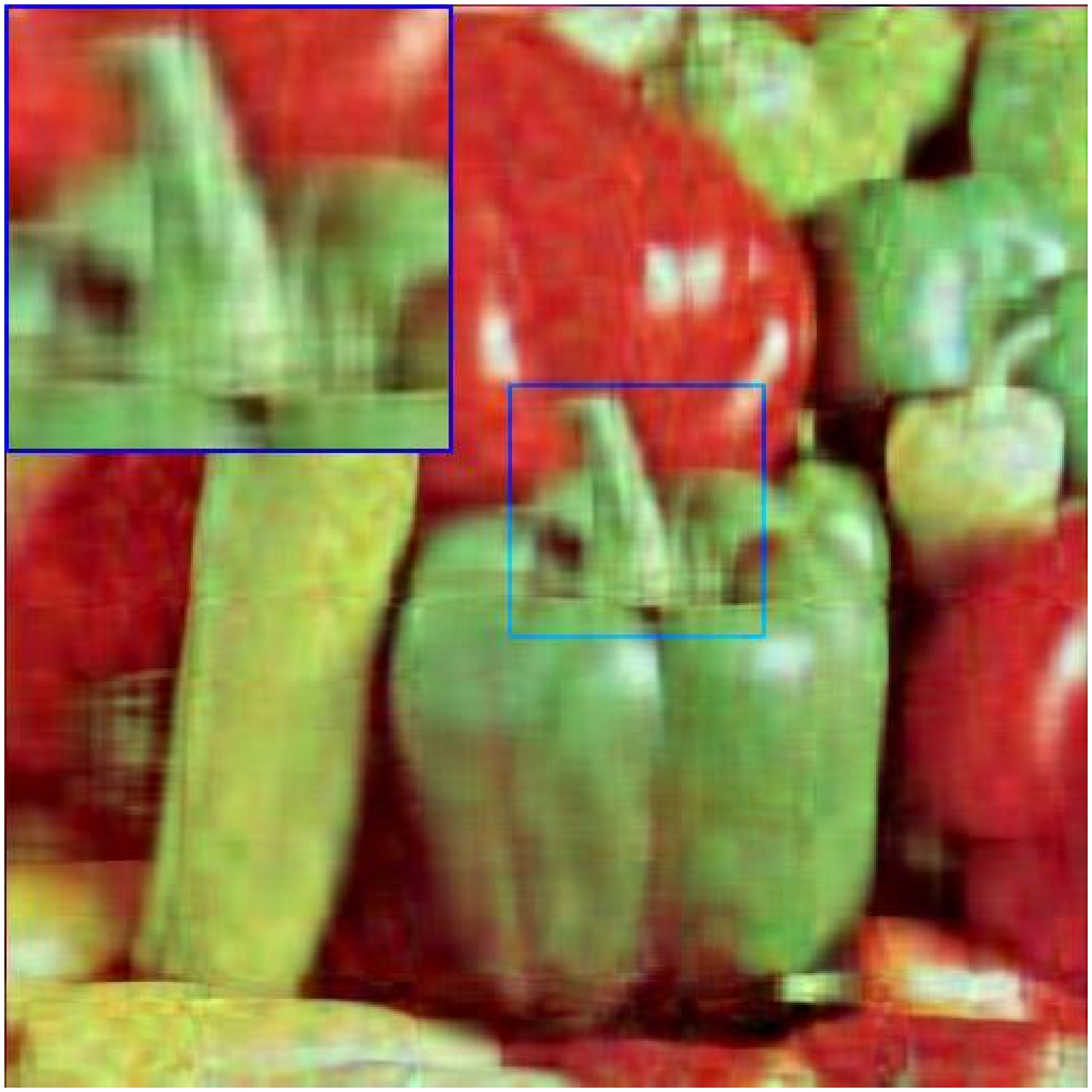}}
\caption{Comparisons of different algorithms for image inpaiting. The right two columns are recovery results of different
algorithms. Specifically, (a) Original $512\times512$ image; (d) The best rank-$(30, 30, 3)$ approximation of
original image; (g) Input to the algorithm (30\% known entries); (b) Recovered image by IHTr-LRTC; (e) Recovered image by IHT-LRTC;
(h) Recovered image by FP-LRTC; (c) Recovered image by TENSOR-HC; (f) Recovered image by HoRPCA;
(i) Recovered image by N-way-E.}
\end{figure}

\begin{figure}[H]
\centering
\subfigure[]{\includegraphics[height=4cm]{org_zoom.eps}} \qquad
\subfigure[]{\includegraphics[height=4cm]{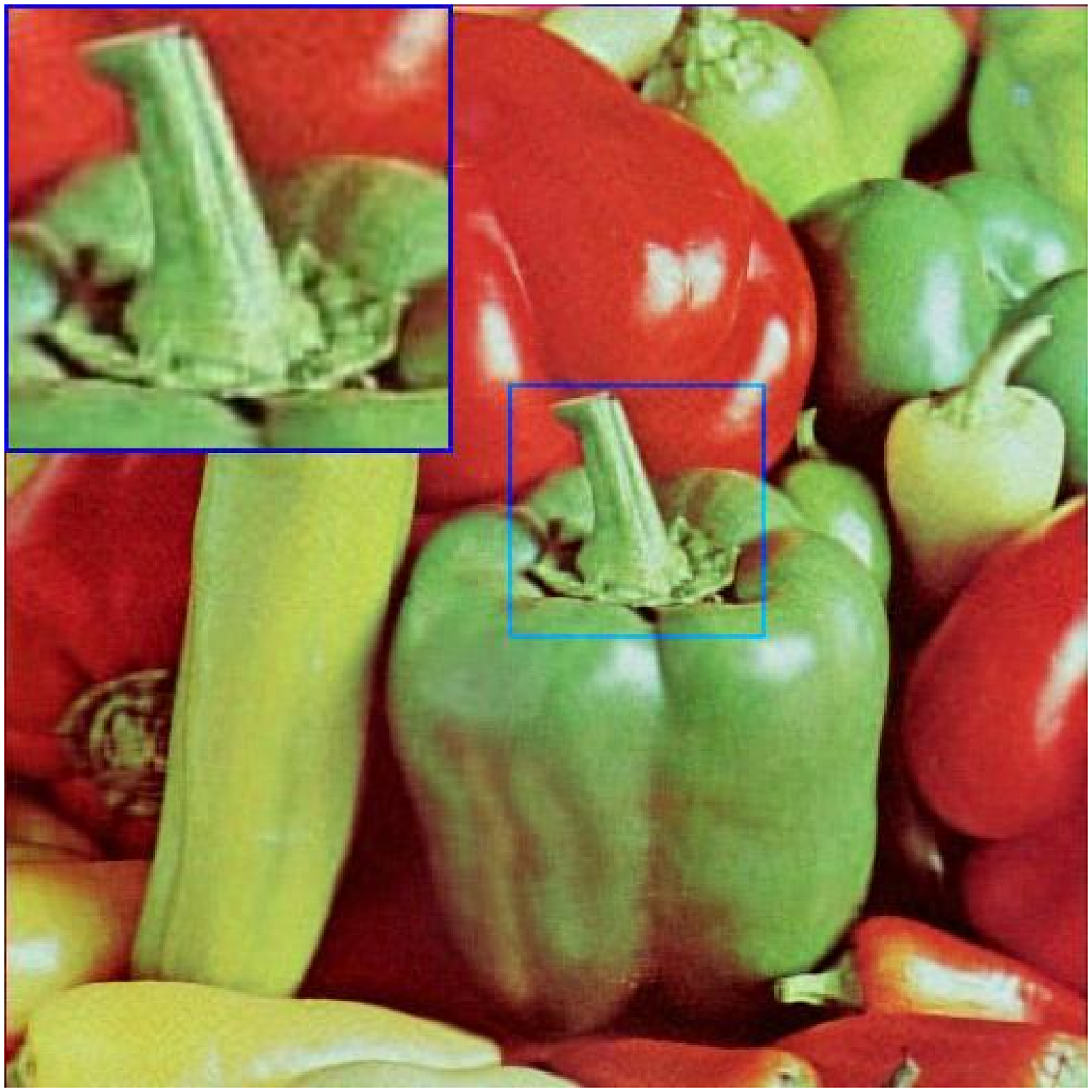}}~
\subfigure[]{\includegraphics[height=4cm]{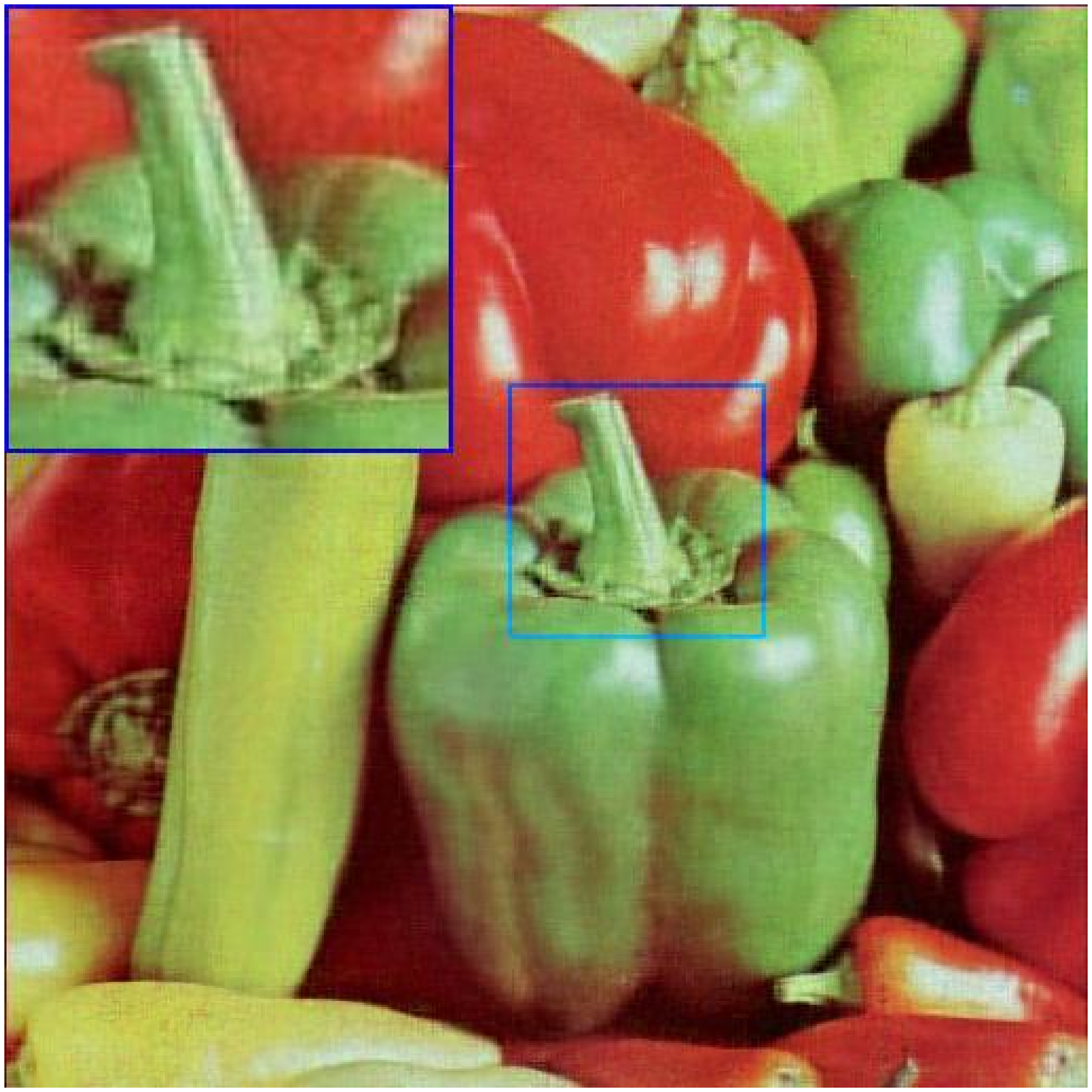}} \\

\subfigure[]{\includegraphics[height=4cm]{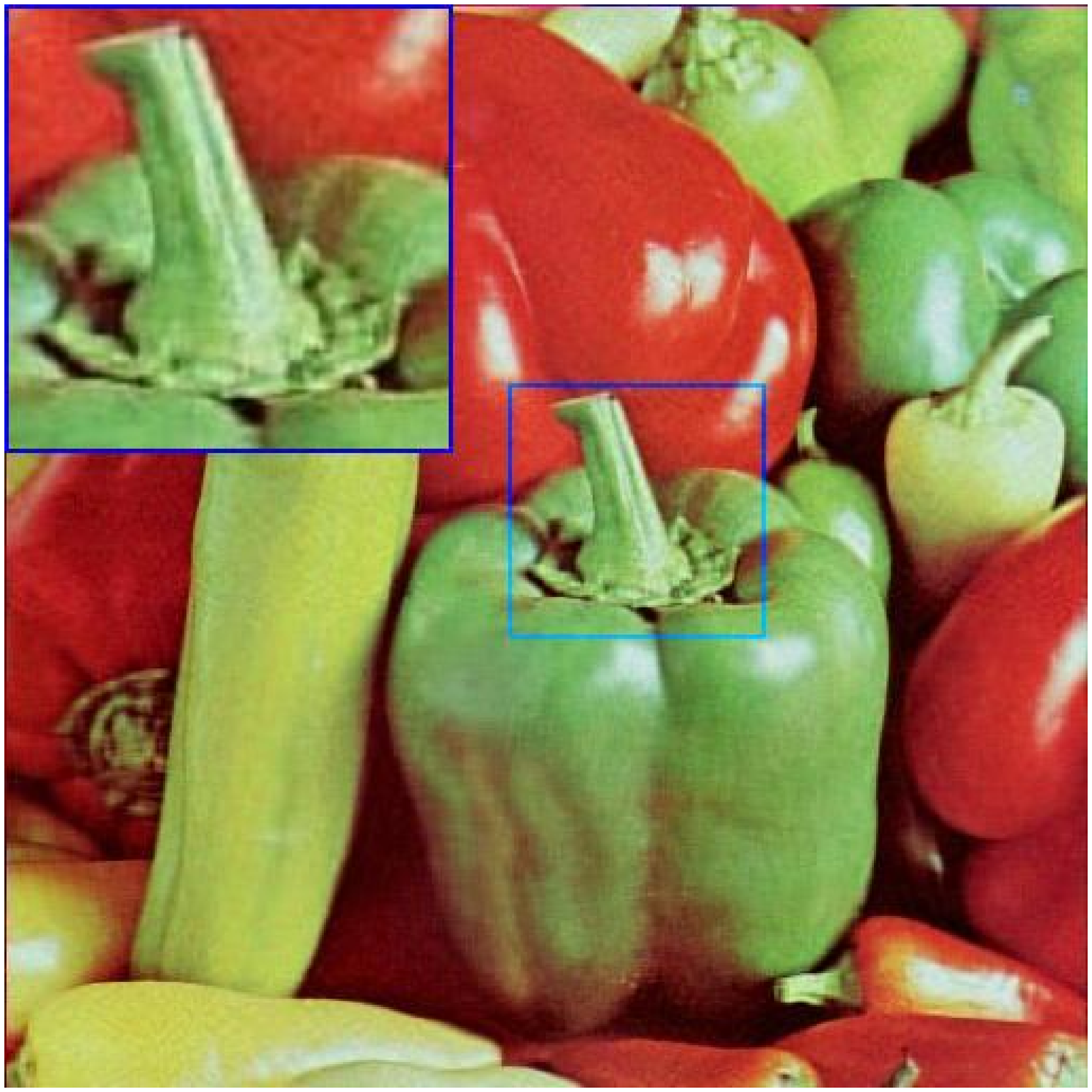}} \qquad
\subfigure[]{\includegraphics[height=4cm]{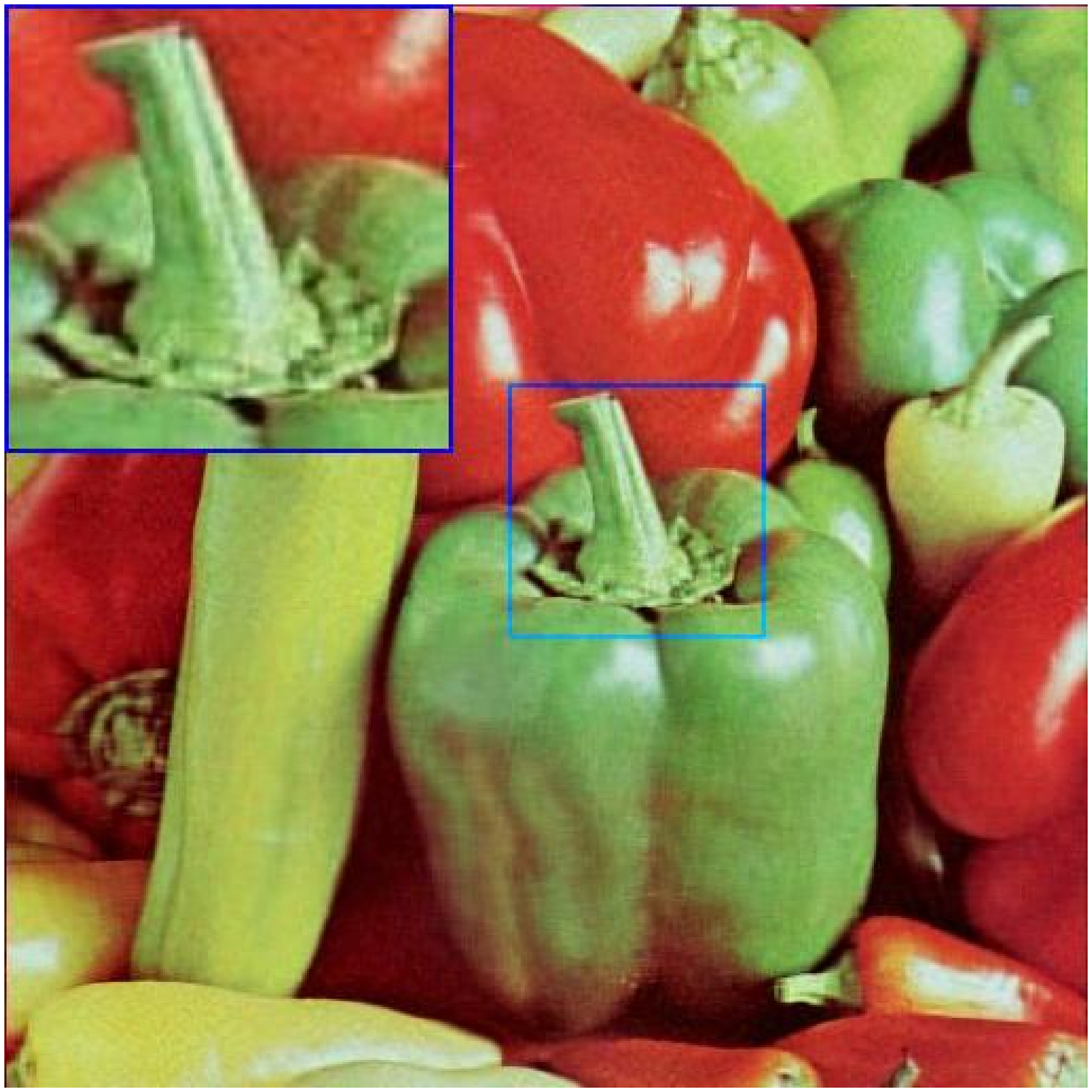}}~
\subfigure[]{\includegraphics[height=4cm]{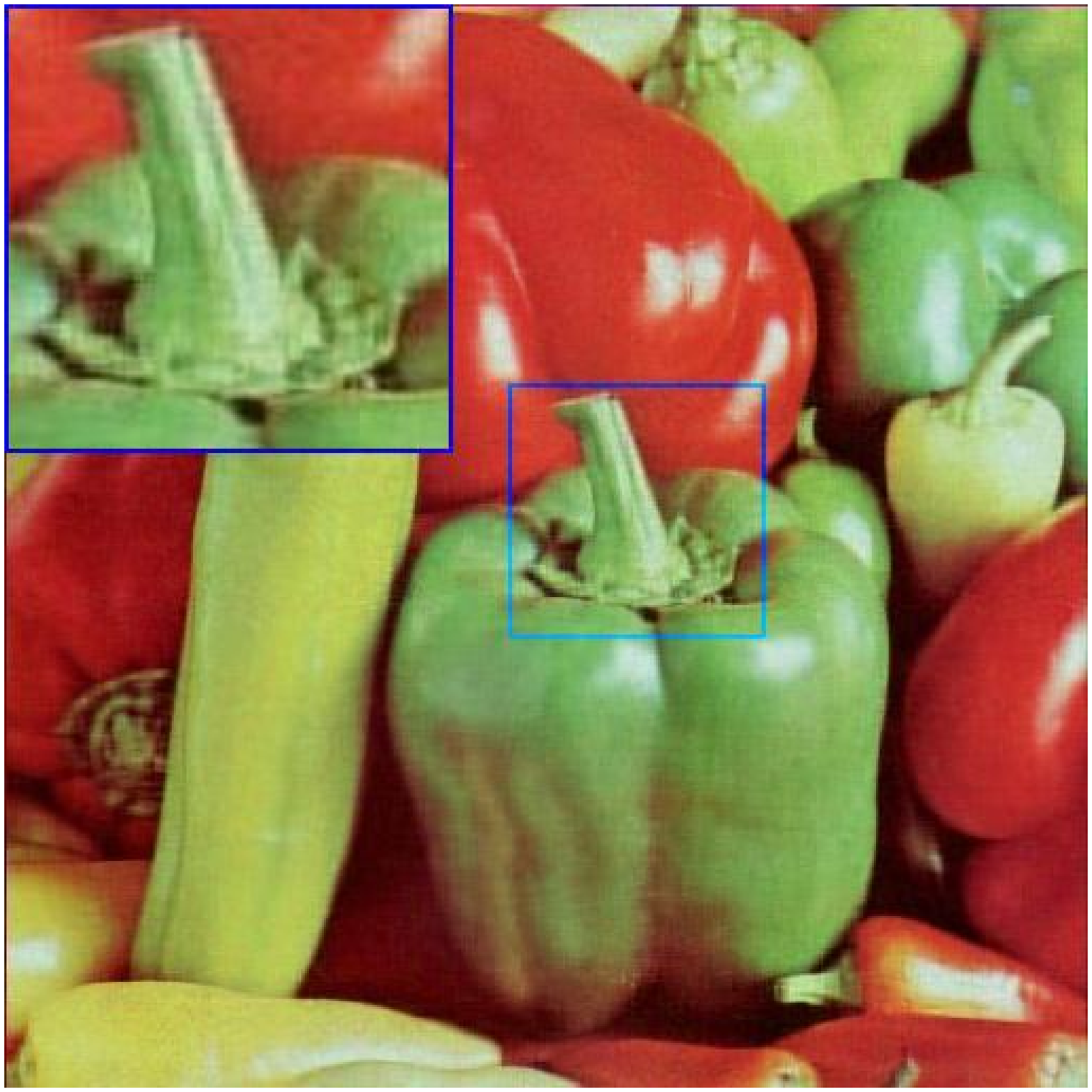}} \\

\subfigure[]{\includegraphics[height=4cm]{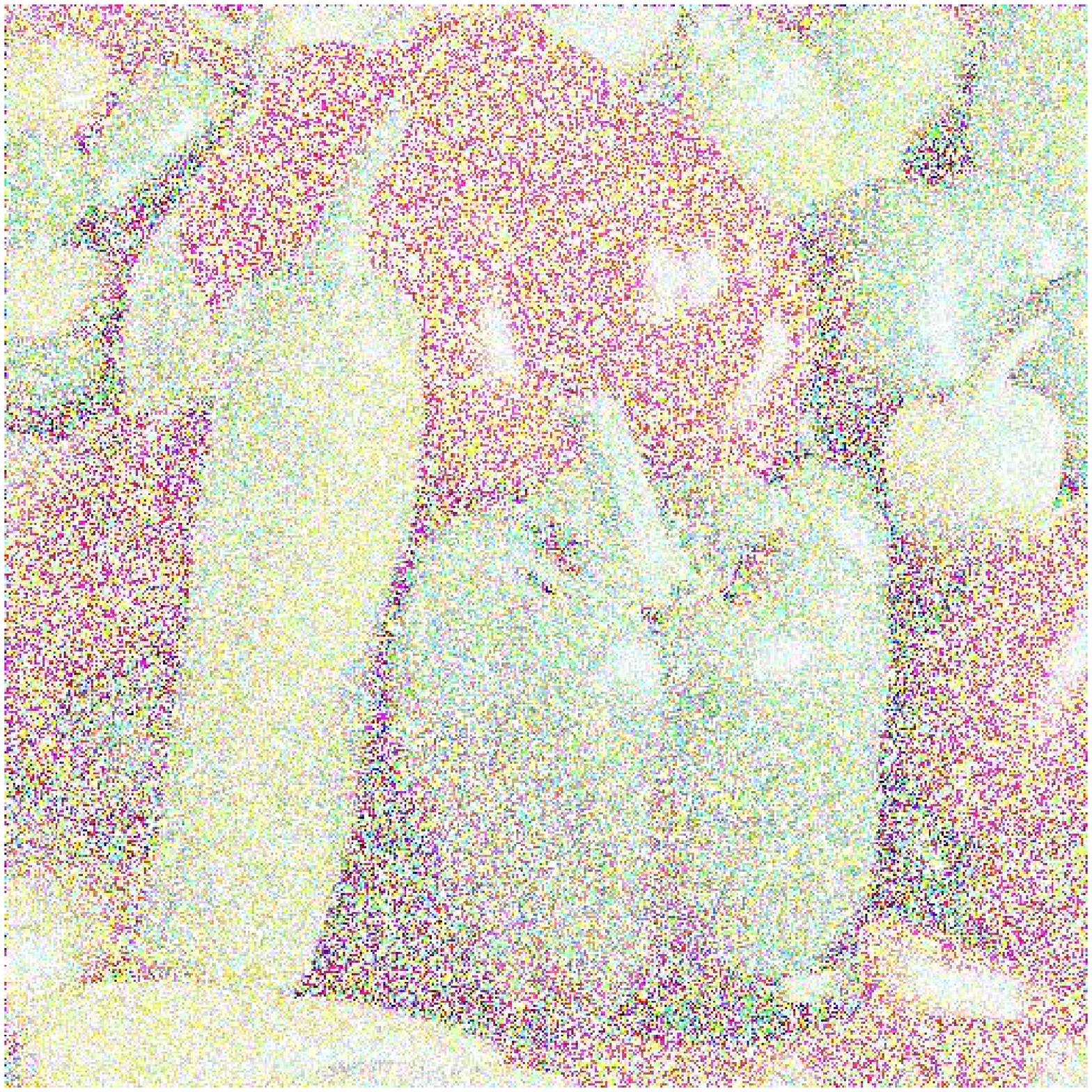}} \qquad
\subfigure[]{\includegraphics[height=4cm]{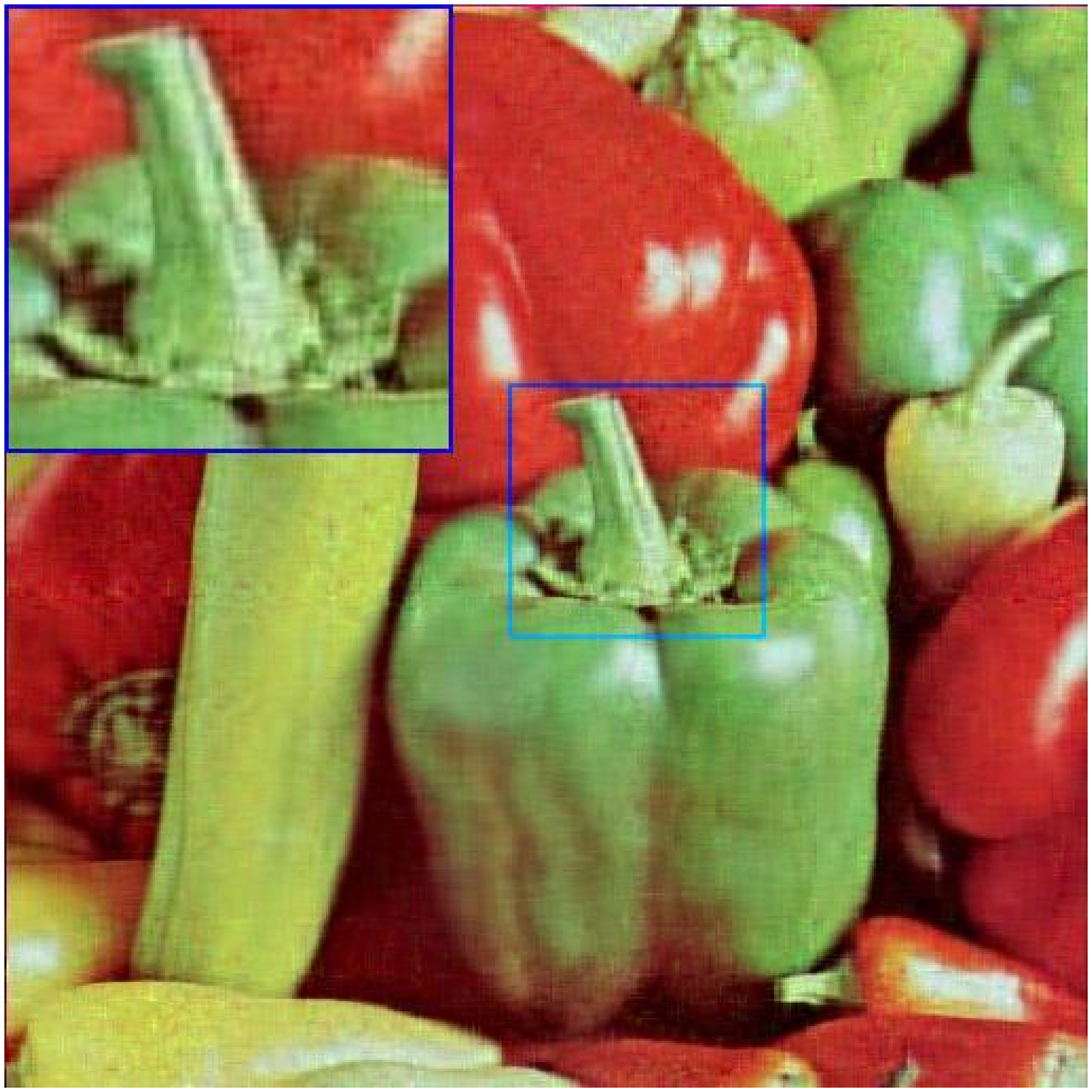}}~
\subfigure[]{\includegraphics[height=4cm]{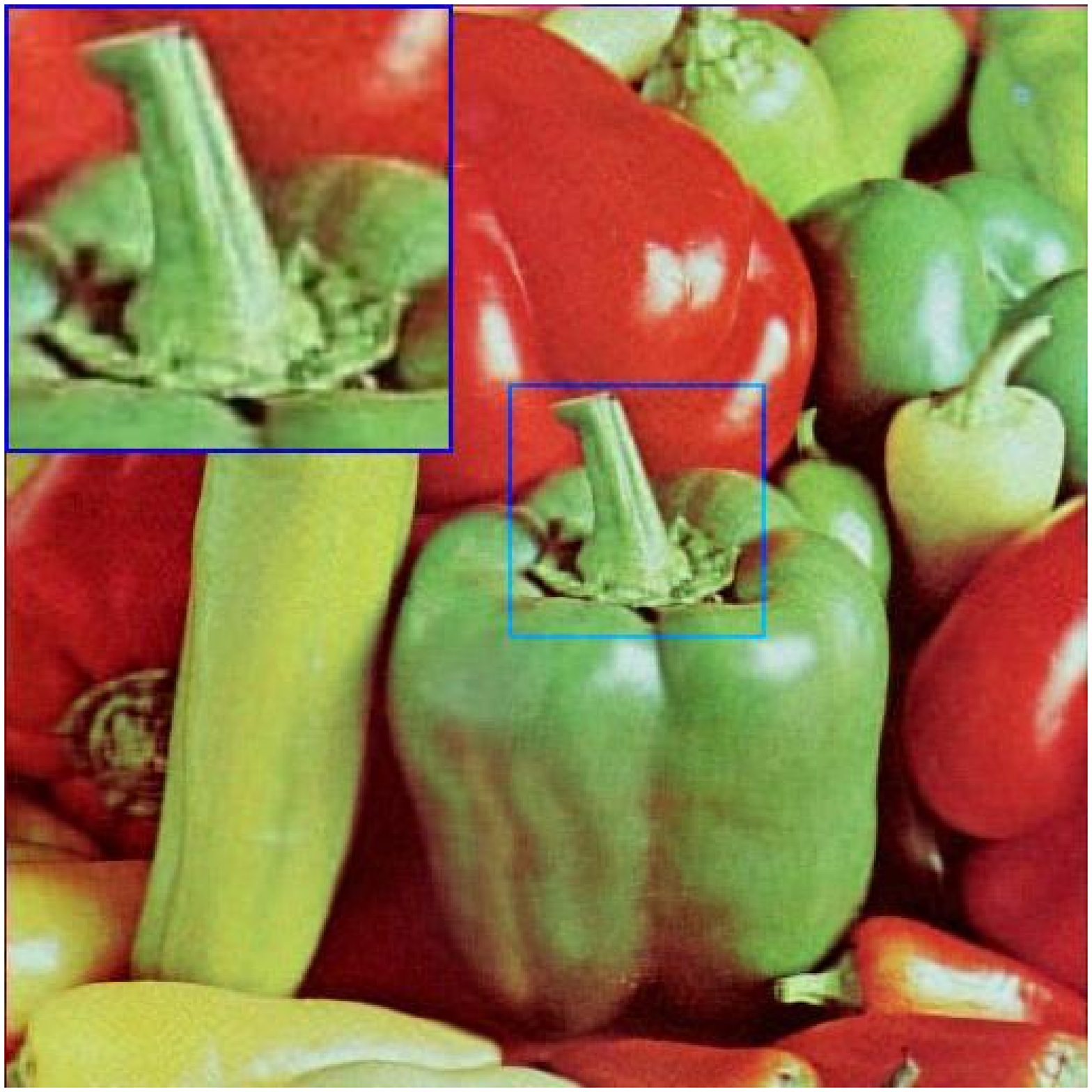}}
\caption{Comparisons of different algorithms for image inpaiting. The right two columns are recovery results of different
algorithms. Specifically, (a) Original $512\times512$ image; (d) The best rank-$(100, 100, 3)$ approximation of
original image; (g) Input to the algorithm (30\% known entries); (b) Recovered image by IHTr-LRTC; (e) Recovered image by IHT-LRTC;
(h) Recovered image by FP-LRTC; (c) Recovered image by TENSOR-HC; (f) Recovered image by HoRPCA;
(i) Recovered image by N-way-E.}
\end{figure}

\begin{table}[H]
\centering \tabcolsep 8pt
\small{
\begin{tabular}{|llcr|llcr|}
\multicolumn{8}{c}{\small{\textbf{Table 3.} Numerical results of different algorithms for image inpainting}} \vspace{1mm} \\
\hline
\small{Algorithm} & \small{iter} & \small{rel.err} & \small{time(s)} & \small{Algorithm} & \small{iter} & \small{rel.err} & \small{time(s)} \\
\hline
\multicolumn{8}{|c|}{\footnotesize{$size=512\times512\times3$, $sr=0.3$}}  \\
\hline
\multicolumn{4}{|c|}{\footnotesize{${\bm r}=(30,30,3)$}} & \multicolumn{4}{c|}{\footnotesize{${\bm r}=(100,100,3)$}} \\
\hline
\footnotesize{IHTr-LRTC}	 &\footnotesize{187}&\footnotesize{1.06e-7}&\footnotesize{32.46}& \footnotesize{IHTr-LRTC} &\footnotesize{1733}&\footnotesize{5.02e-7}&\footnotesize{332.05}	 \vspace{-0.8mm}\\	
\footnotesize{IHT-LRTC}	    &\footnotesize{665}&\footnotesize{6.40e-8}&\footnotesize{128.68}& \footnotesize{IHT-LRTC} &\footnotesize{5000}&\footnotesize{7.24e-4}&\footnotesize{998.69}		 \vspace{-0.8mm}\\
\footnotesize{FP-LRTC}	    &\footnotesize{1040}&\footnotesize{5.17e-7}&\footnotesize{454.13}& \footnotesize{FP-LRTC} &\footnotesize{1040}&\footnotesize{5.40e-2}&\footnotesize{629.40}	 \vspace{-0.8mm}\\	
\footnotesize{TENSOR-HC}	 &\footnotesize{144}&\footnotesize{1.81e-7}&\footnotesize{246.65}& \footnotesize{TENSOR-HC} &\footnotesize{775}&\footnotesize{4.58e-2}&\footnotesize{1337.67}		 \vspace{-0.8mm} \\
\footnotesize{HoRPCA}	 &\footnotesize{138}&\footnotesize{3.33e-7}&\footnotesize{93.69}& \footnotesize{HoRPCA} &\footnotesize{1000}&\footnotesize{4.55e-2}&\footnotesize{1198.55}	 \vspace{-0.8mm}\\
\footnotesize{N-way-E}	 &\footnotesize{107}&\footnotesize{4.72e-7}&\footnotesize{121.64}& \footnotesize{N-way-E} &\footnotesize{618}&\footnotesize{1.34e-4}&\footnotesize{1336.11}	 \vspace{0mm}\\
\hline
\end{tabular}}
\end{table}

\section{Conclusions }
In this paper, we considered a new alternative recovery model `MnRA' and proposed an appropriate iterative hard
thresholding algorithm to solve it with giving upper bound of $n$-rank in advance. The convergence analysis of the proposed
algorithm was also presented. By using an effective heuristic of determining $n$-rank, we can also apply the proposed algorithm
to solve MnRA with unknown $n$-rank in advance. Some preliminary numerical results on LRTC were reported.
Through the theoretical analysis and numerical experiments, we can draw some encouraging conclusions:
\begin{itemize}
\item The model of MnRA in this paper is creative in low $n$-rank tensor recovery. MnRA includes both noiseless and noisy case. And although the model needs the $n$-rank of the original data as prior information, we have proposed a heuristic for determining $n$-rank and this method turned to be efficient.

\item The iterative hard thresholding algorithm proposed in this paper is easy to implement. It has a very simple
      iterative scheme and only one parameter $\tau$, which can be easily estimated from theoretical analysis
      and can be chosen broadly in practice.

\item The iterative hard thresholding algorithm is extremely fast. Actually, the iterative sequence generated by the proposed
      algorithm is globally linearly convergent with the rate $\frac{1}{2}$ for the noiseless case. In our numerical experiments, these theoretical results can be confirmed.

\item IHTr-LRTC and IHT-LRTC are still effective for the tensor with high $n$-rank. Thus, they may have wider
      applications in practice.
\end{itemize}

It is interesting to investigate how to determine $n$-rank more effectively in practice. We believe that the
iterative hard thresholding algorithm combining with the appropriate method for predicting $n$-rank can be used
to solve more general tensor optimization problems. Moreover, the nonconvex sparse optimization problems and the
related algorithms in vector or matrix space have been widely discussed in the literature \citep{mzy2013, yzy2013}.
It is worth investigating how to apply the iterative hard thresholding algorithm to solve the nonconvex model
in the tensor space.

\section*{Acknowledgements}

We would like to thank Silvia Gandy for sending us the code of ADM-TR(E), and thank Marco Signoretto
for sending us the code of TENSOR-HC. This work was partially supported by the National Natural
Science Foundation of China (Grant No. 11171252 and No.11201332).

\end{document}